\documentclass[reqno,12pt]{amsart}

\usepackage[utf8]{inputenc}

\usepackage{etoolbox}

\makeatletter
\let\old@tocline\@tocline
\let\section@tocline\@tocline
\newcommand{\subsection@dotsep}{4.5}
\newcommand{\subsubsection@dotsep}{4.5}
\patchcmd{\@tocline}
  {\hfil}
  {\nobreak
     \leaders\hbox{$\m@th
        \mkern \subsection@dotsep mu\hbox{.}\mkern \subsection@dotsep mu$}\hfill
     \nobreak}{}{}
\let\subsection@tocline\@tocline
\let\@tocline\old@tocline

\patchcmd{\@tocline}
  {\hfil}
  {\nobreak
     \leaders\hbox{$\m@th
        \mkern \subsubsection@dotsep mu\hbox{.}\mkern \subsubsection@dotsep mu$}\hfill
     \nobreak}{}{}
\let\subsubsection@tocline\@tocline
\let\@tocline\old@tocline

\let\old@l@subsection\l@subsection
\let\old@l@subsubsection\l@subsubsection

\def\@tocwriteb#1#2#3{%
  \begingroup
    \@xp\def\csname #2@tocline\endcsname##1##2##3##4##5##6{%
      \ifnum##1>\c@tocdepth
      \else \sbox\z@{##5\let\indentlabel\@tochangmeasure##6}\fi}%
    \csname l@#2\endcsname{#1{\csname#2name\endcsname}{\@secnumber}{}}%
  \endgroup
  \addcontentsline{toc}{#2}%
    {\protect#1{\csname#2name\endcsname}{\@secnumber}{#3}}}%

\newlength{\@tocsectionindent}
\newlength{\@tocsubsectionindent}
\newlength{\@tocsubsubsectionindent}
\newlength{\@tocsectionnumwidth}
\newlength{\@tocsubsectionnumwidth}
\newlength{\@tocsubsubsectionnumwidth}
\newcommand{\settocsectionnumwidth}[1]{\setlength{\@tocsectionnumwidth}{#1}}
\newcommand{\settocsubsectionnumwidth}[1]{\setlength{\@tocsubsectionnumwidth}{#1}}
\newcommand{\settocsubsubsectionnumwidth}[1]{\setlength{\@tocsubsubsectionnumwidth}{#1}}
\newcommand{\settocsectionindent}[1]{\setlength{\@tocsectionindent}{#1}}
\newcommand{\settocsubsectionindent}[1]{\setlength{\@tocsubsectionindent}{#1}}
\newcommand{\settocsubsubsectionindent}[1]{\setlength{\@tocsubsubsectionindent}{#1}}

\renewcommand{\l@section}{\section@tocline{1}{\@tocsectionvskip}{\@tocsectionindent}{}{\@tocsectionformat}}%
\renewcommand{\l@subsection}{\subsection@tocline{2}{\@tocsubsectionvskip}{\@tocsubsectionindent}{}{\@tocsubsectionformat}}%
\renewcommand{\l@subsubsection}{\subsubsection@tocline{3}{\@tocsubsubsectionvskip}{\@tocsubsubsectionindent}{}{\@tocsubsubsectionformat}}%
\newcommand{\@tocsectionformat}{}
\newcommand{\@tocsubsectionformat}{}
\newcommand{\@tocsubsubsectionformat}{}
\expandafter\def\csname toc@1format\endcsname{\@tocsectionformat}
\expandafter\def\csname toc@2format\endcsname{\@tocsubsectionformat}
\expandafter\def\csname toc@3format\endcsname{\@tocsubsubsectionformat}
\newcommand{\settocsectionformat}[1]{\renewcommand{\@tocsectionformat}{#1}}
\newcommand{\settocsubsectionformat}[1]{\renewcommand{\@tocsubsectionformat}{#1}}
\newcommand{\settocsubsubsectionformat}[1]{\renewcommand{\@tocsubsubsectionformat}{#1}}
\newlength{\@tocsectionvskip}
\newcommand{\settocsectionvskip}[1]{\setlength{\@tocsectionvskip}{#1}}
\newlength{\@tocsubsectionvskip}
\newcommand{\settocsubsectionvskip}[1]{\setlength{\@tocsubsectionvskip}{#1}}
\newlength{\@tocsubsubsectionvskip}
\newcommand{\settocsubsubsectionvskip}[1]{\setlength{\@tocsubsubsectionvskip}{#1}}

\patchcmd{\tocsection}{\indentlabel}{\makebox[\@tocsectionnumwidth][l]}{}{}
\patchcmd{\tocsubsection}{\indentlabel}{\makebox[\@tocsubsectionnumwidth][l]}{}{}
\patchcmd{\tocsubsubsection}{\indentlabel}{\makebox[\@tocsubsubsectionnumwidth][l]}{}{}

\newcommand{\@sectypepnumformat}{}
\renewcommand{\contentsline}[1]{%
  \expandafter\let\expandafter\@sectypepnumformat\csname @toc#1pnumformat\endcsname%
  \csname l@#1\endcsname}
\newcommand{\@tocsectionpnumformat}{}
\newcommand{\@tocsubsectionpnumformat}{}
\newcommand{\@tocsubsubsectionpnumformat}{}
\newcommand{\setsectionpnumformat}[1]{\renewcommand{\@tocsectionpnumformat}{#1}}
\newcommand{\setsubsectionpnumformat}[1]{\renewcommand{\@tocsubsectionpnumformat}{#1}}
\newcommand{\setsubsubsectionpnumformat}[1]{\renewcommand{\@tocsubsubsectionpnumformat}{#1}}
\renewcommand{\@tocpagenum}[1]{%
  \hfill {\mdseries\@sectypepnumformat #1}}

\let\oldappendix\appendix
\renewcommand{\appendix}{%
  \leavevmode\oldappendix%
  \addtocontents{toc}{%
    \protect\settowidth{\protect\@tocsectionnumwidth}{\protect\@tocsectionformat\sectionname\space}%
    \protect\addtolength{\protect\@tocsectionnumwidth}{2em}}%
}
\makeatother

\makeatletter
\settocsectionnumwidth{2em}
\settocsubsectionnumwidth{2.5em}
\settocsubsubsectionnumwidth{3em}
\settocsectionindent{1pc}%
\settocsubsectionindent{\dimexpr\@tocsectionindent+\@tocsectionnumwidth}%
\settocsubsubsectionindent{\dimexpr\@tocsubsectionindent+\@tocsubsectionnumwidth}%
\makeatother

\settocsectionvskip{10pt}
\settocsubsectionvskip{0pt}
\settocsubsubsectionvskip{0pt}
    
\settocsectionformat{\bfseries}
\settocsubsectionformat{\mdseries}
\settocsubsubsectionformat{\mdseries}
\setsectionpnumformat{\bfseries}
\setsubsectionpnumformat{\mdseries}
\setsubsubsectionpnumformat{\mdseries}

\let\oldtableofcontents\tableofcontents
\renewcommand{\tableofcontents}{%
  \vspace*{-\linespacing}
  \oldtableofcontents}

\usepackage[bookmarks=true, bookmarksopen=true,%
bookmarksdepth=3,bookmarksopenlevel=2,%
colorlinks=true,%
linkcolor=black,%
citecolor=blue,%
filecolor=blue,%
menucolor=blue,%
urlcolor=blue]{hyperref}
\usepackage{amsmath, amscd, amssymb,amsfonts,amsthm,amscd,graphicx}
\usepackage{array, float, mathrsfs, mathtools, stmaryrd, youngtab}

\usepackage{cleveref}
\usepackage[rgb]{xcolor} 
\definecolor{mygreen}{rgb}{0,0.7,0.3}
\definecolor{myblue}{rgb}{0,0.50,1.20}
\definecolor{myorange}{rgb}{1,0.5,0.1}
\definecolor{fillred}{rgb}{1,0.9,0.9}
\definecolor{fillgreen}{rgb}{0.9,1,0.9}

\usepackage[all]{xy}
\usepackage{tikz}
\usepackage{tikz-cd}
\usetikzlibrary{arrows,chains,shapes,matrix,positioning,scopes} 
\usetikzlibrary{decorations.markings}
\tikzstyle arrowstyle=[scale=1.1]
\tikzstyle directed=[postaction={decorate,decoration={markings,
    mark=at position 1 with {\arrow[arrowstyle]{latex}}}}]
\tikzstyle reverse directed=[postaction={decorate,decoration={markings,
    mark=at position .45 with {\arrowreversed[arrowstyle]{latex};}}}]
    
\newcommand\qarrow[2]{\draw[->,shorten >=2pt,shorten <=2pt] (#1) -- (#2) [thick];} 
\newcommand\qdarrow[2]{\draw[->,dashed,shorten >=2pt,shorten <=2pt] (#1) -- (#2) [thick];} 

\newcommand\rqstdarrow[2]{\draw[->,dashed,shorten >=2pt,shorten <=4pt,red] (#1) -- (#2) [thick];}

\newcommand\rqarrow[2]{\draw[->,shorten >=6pt,shorten <=6pt,red] (#1) -- (#2) [thick];}

\newtheorem{theorem}{Theorem}[section]

\newtheorem{lemma}[theorem]{Lemma}
\newtheorem{proposition}[theorem]{Proposition}
\newtheorem{corollary}[theorem]{Corollary}

\newtheorem{conjecture}[theorem]{Conjecture}
\newtheorem{theorem-definition}[theorem]{Theorem-Definition}
\newtheorem{theorem-construction}[theorem]{Theorem-Construction}
\newtheorem{lemma-definition}[theorem]{Lemma--Definition}
\newtheorem{lemma-construction}[theorem]{Lemma--construction}
\newtheorem{definition}[theorem]{Definition}
\theoremstyle{definition}

\newtheorem{remark}[theorem]{Remark}
\newtheorem{example}[theorem]{Example}

\newcommand{\bg}{\begin{equation}\begin{gathered}}
\newcommand{\eg}{\end{gathered}\end{equation}}

\newcommand{\g}{\mathfrak{g}}

\newcommand{\uqg}{\mathcal{U}_q(\mathfrak{g})}
\newcommand{\tqg}{\widetilde{\mathcal{U}}_q(\mathfrak{g})}

\newcommand{\Z}{\mathbb{Z}}

\newcommand{\Uq}{{\mathcal U}_{q}}
\newcommand{\bUq}{{\bf U}_q(\mathfrak{g})}
\newcommand{\tbUq}{\widetilde{\bf U}_q(\mathfrak{g})}

\newcommand{\bF}{{\bf F}}
\newcommand{\bE}{{\bf E}}
\newcommand{\bK}{{\bf K}}
\newcommand{\tbK}{\widetilde{\bf K}}

\newcommand{\old}[1]{}

\newcommand{\Q}{{\mathbb Q}}

\newcommand{\A}{{\mathsf{A}}}
\newcommand{\B}{{\mathsf{B}}}
\newcommand{\G}{\mathsf{G}}
\newcommand{\U}{\mathsf{U}}
\renewcommand{\H}{\mathsf{H}}

\newcommand{\bS}{{\mathbb{S}}}
\newcommand{\lms}{\longmapsto}
\newcommand{\lra}{\longrightarrow}

\newcommand{\be}{\begin{equation}}
\newcommand{\ee}{\end{equation}}
\newcommand{\bt}{\begin{theorem}}
\newcommand{\et}{\end{theorem}}
\newcommand{\bd}{\begin{definition}}
\newcommand{\ed}{\end{definition}}
\newcommand{\bp}{\begin{proposition}}
\newcommand{\ep}{\end{proposition}}

\newcommand{\bl}{\begin{lemma}}
\newcommand{\el}{\end{lemma}}
\newcommand{\bc}{\begin{corollary}}
\newcommand{\ec}{\end{corollary}}
\newcommand{\bcon}{\begin{conjecture}}
\newcommand{\econ}{\end{conjecture}}
\newcommand{\la}{\label}
\newcommand{\OPW}{\mathcal{O}_q(\mathscr{P}_{\G, \odot})^W}

\title[Cluster nature of quantum groups]{Cluster nature of quantum groups}
\author{Linhui Shen}
\date{}
\address{Department of Mathematics, Michigan State University, East Lansing, MI} 
\email{linhui@math.msu.edu}

\begin{document}

\maketitle

\begin{abstract} 
We present a rigid cluster model to realize the quantum group ${\bf U}_q(\mathfrak{g})$ for $\mathfrak{g}$ of type ADE. That is, we prove that there is a natural Hopf algebra isomorphism from  the quantum group ${\bf U}_q(\mathfrak{g})$ to a quotient algebra of the Weyl group invariants of the Fock-Goncharov quantum cluster algebra $\mathcal{O}_q(\mathscr{P}_{\G,\odot})$. 
 By applying the quantum duality of cluster algebras, we show that ${\bf U}_q(\mathfrak{g})$ admits a natural basis $\overline{\bf \Theta}$ whose structural coefficients are in $\mathbb{N}[q^{\frac{1}{2}}, q^{-\frac{1}{2}}]$. The basis $\overline{\bf \Theta}$ satisfies an invariance property under Lusztig's braid group action, the Dynkin automorphisms, and the star anti-involution.  
\end{abstract}

\tableofcontents

\section{Introduction}
The quantum group $\mathcal{U}_q(\mathfrak{g})$, introduced by Drinfeld \cite{Dr} and Jimbo \cite{Ji} in 1985, continues to be a central topic in representation theory, knot theory, statistical physics, etc. Cluster algebras, introduced by Fomin and Zelevinsky \cite{FZI}, are a class of commutative algebras associated with integer skew-symmetrizable matrices and their mutations. Since its inception, the
rapid developments of cluster theory have found numerous exciting applications in representation theory, quantum geometry, contact topology, etc. The quantization of {\it cluster Poisson varieties} and their principal series representations has been systematically studied by Fock and Goncharov \cite{FGrep}. One of the primary motivations for studying cluster algebras is understanding several remarkable features of quantum groups, including {\it total positivity} and {\it canonical bases}. We realize this goal in the present paper by providing a rigid cluster model for the quantum groups of type $ADE$.

\smallskip 

Let $\mathbb{L}=\mathbb{Z}[q^{\frac{1}{2}}, q^{-\frac{1}{2}}]$ be the ring of Laurent polynomials in the variable $q^{\frac{1}{2}}$ with integer coefficients and let $\mathbb{K}= \mathbb{Q}(q^{\frac{1}{2}})$ be its quotient field. 
The quantum group $\mathcal{U}_q(\mathfrak{g})$ is a quantized universal enveloping algebra associated with a Lie algebra $\mathfrak{g}$. In this paper, we assume that $\mathfrak{g}$ is of type ADE. The quantum group $\mathcal{U}_q(\mathfrak{g})$ carries an integral form ${\bf U}_q(\mathfrak{g})$, defined as the $\mathbb{L}$-linear span of a rescaled PBW basis (cf. Section \ref{Sec.quan.group}). The semiclassical limit of ${\bf U}_q(\mathfrak{g})$ recovers the coordinate ring of the Poisson Lie dual group.

\smallskip
Let $\G$ be a split algebraic group over $\mathbb{Q}$  such that  ${\rm Lie}\,\G=\mathfrak{g}$. We assume that $\G$ has trivial center and let $\tilde{\G}$ be its universal cover. 
Let $\bS$ be an oriented topological surface with punctures and marked boundary points.
In their seminal work \cite{FGteich}, Fock and Goncharov introduced a pair $(\mathscr{A}_{\tilde{\G},\bS}, \mathscr{X}_{\G, \bS})$ of moduli spaces, as an algebraic-geometric avatar of higher Teichm\"uller theory. In {\it loc.cit.}, they further show that the pair $(\mathscr{A}_{{\rm SL}_{n+1},\bS}, \mathscr{X}_{{\rm PGL}_{n+1}, \bS})$ carries a natural cluster structure. The same result has been extended by Ian Le \cite{Le} to classical groups, based on a case by case construction. In \cite{GS3}, Goncharov and the author give a uniform construction of the desired cluster structures for all semisimple groups.
\smallskip

 The papers \cite{GS1, GS3} introduce a decorated character variety $\mathscr{P}_{\G, \bS}$ as an enhanced version of Fock and Goncharov's moduli space $\mathscr{X}_{\G,\bS}$  by assigning a {\it pinning} to every boundary interval. The importance of the pinning is reflected at least in two aspects. First, when $\bS$ has boundary, the pinning data increases the dimension of $\mathscr{X}_{\G,\bS}$  by adding new frozen cluster variables, so that the pair $(\mathscr{P}_{\G, \bS}, \mathscr{A}_{\widetilde{\G},\bS})$ forms a {\it cluster ensemble}.
Second, the pinnings allow us to amalgamate different surfaces along boundary intervals, which eventually makes the moduli spaces into a geometric avatar of TQFT. The coordinate ring of  $\mathscr{P}_{\G, \bS}$ is a {\it cluster Poisson algebra} (\cite[Th.1]{Sh2}) and can be quantized to a Fock-Goncharov {\it quantum cluster algebra}  $\mathcal{O}_q(\mathscr{P}_{\G, \bS})$.

\smallskip

In this paper, we focus on the case when $\bS=\odot$ is a once-punctured disk with two marked points. As proved in \cite{GS2, GS3}, the Weyl group acts on $\mathcal{O}_q(\mathscr{P}_{\G, \odot})$ via cluster automorphisms. Denote by $\mathcal{O}_q(\mathscr{P}_{\G, \odot})^W$ the subalgebra of $W$-invariants of $\mathcal{O}_q(\mathscr{P}_{\G, \odot})$. 
The {\it outer monodromies} $\mathbb{O}_1, \ldots, \mathbb{O}_r$ are particular {\it Casimir} elements in  $\mathcal{O}_q(\mathscr{P}_{\G, \bS})^W$ that correspond to simple positive roots of $\G$. Let $\mathcal{I}$ be the ideal generated by $\mathbb{O}_i-1$ for $i=1,\ldots, r.$ 
The quotient algebra $\mathcal{O}_q(\mathscr{P}_{\G, \odot})^W{\slash}\mathcal{I}$ is naturally a Hopf algebra, obtained via amalgamation of punctured disks.

\smallskip

Our main result is as follows.
\bt
\label{main1}
There is a natural Hopf algebra isomorphism
\[
\kappa: ~ {\bf U}_q(\mathfrak{g})\stackrel{\sim}{\lra} \mathcal{O}_q(\mathscr{P}_{\G, \odot})^W{\Big \slash} \mathcal{I}.
\]
\et 
 Theorem \ref{main1} is proven in Section \ref{sec5.4}.
Below we include a few historical comments. 
\begin{itemize}
\item The idea that relates the quantum group $\mathcal{U}_q(\mathfrak{g})$ to the character variety of the marked once punctured disk is originally due to Fock and Goncharov in 2006. 

\smallskip

\item Following a suggestion of Fock, Schrader and Shapiro \cite{SS} construct a concrete injective algebra homomorphism from the quantum group $\mathcal{U}_q(\mathfrak{sl}_{n+1})$ into a quantum torus algebra associated with the Fock-Goncharov quiver $\mathcal{D}_n$ for $\mathscr{A}_{{\rm SL}_{n+1},\odot}$. With the help from the author, Schrader and Shapiro further show that the image of $\mathcal{U}_q(\mathfrak{sl}_{n+1})$ is contained in the quantum cluster algebra for $\mathcal{D}_n$. In this way, they obtain, for the first time,  an embedding $\kappa$ of the quantum group $\mathcal{U}_q(\mathfrak{sl}_{n+1})$ into a quantum cluster algebra. However, their embedding is not surjective and the image of $\kappa$ is not known. The paper \cite{SS} further obtains a description of the coproduct of $\mathcal{U}_q(\mathfrak{sl}_{n+1})$ in terms of the corresponding quantum cluster algebra associated with the marked twice punctured disk, and shows a remarkable correspondence of the action of the $R$-matrix with the half-Dehn twist of the twice punctured disk.

\smallskip 

\item For other simple Lie algebras $\mathfrak{g}$, Ivan Ip \cite{Ip} constructed an embedding of the quantum group $\mathcal{U}_q(\mathfrak{g})$ into a {\it quantum torus algebra} associated with a certain quiver $\mathcal{D}_{\mathfrak{g}}$. However, it is not clear whether the image of $\mathcal{U}_q(\mathfrak{g})$ is in the quantum cluster algebra associated with $\mathcal{D}_\mathfrak{g}$.

\smallskip

\item The paper \cite{GS1} introduces, for each marked  point on $\bS$, the potential functions $\mathcal{W}_i$ and the $h$-distance functions $\mathcal{K}_i$ on $\mathscr{P}_{\G, \bS}$. When $\bS$ is a disk, the tropicalizations of $\mathcal{W}_i$ and $\mathcal{K}_i$ give rise to a natural parametrization of  the Mirkovi\'c-Vilonen basis for the tensor invariants of the representations of the Langlands dual group $\G^L$. 

\smallskip 

\noindent For each marked point on $\bS$, the paper \cite{GS3} defines an natural embedding of $\mathcal{U}_q(\mathfrak{g})$ into $\mathcal{O}_q(\mathscr{P}_{\G, \bS})$,  by sending the generators ${\bf E}_i, {\bf K}_i$ to the quantum cluster promotion of $\mathcal{W}_i$ and $\mathcal{K}_i$ respectively.
When specializing to $\odot$, the paper {\it loc.cit.} gives a natural Hopf algebra embedding $\kappa$ from the quantum group ${\bf U}_q(\mathfrak{g})$ into $\mathcal{O}_q(\mathscr{P}_{\G, \odot})^W{\Big \slash}\mathcal{I}$, and further conjectures that the map $\kappa$ is an isomorphism. For $\mathfrak{sl}_2$, the surjectivity of $\kappa$ can be verified by a direct calculation, since the underlying cluster structure is simply of type $A_1\times A_1$. For any other higher rank $\mathfrak{g}$ beyond $\mathfrak{sl}_2$, the surjectivity of $\kappa$ was not known, and seems impossible to verify via a direct calculation, since the size of the underlying quiver increases tremendously. For example, the ice quiver for $\mathfrak{sl}_4$, as shown on Figure \ref{odotquiver}, has 12 mutable vertices and 6 frozen vertices.

\smallskip

\item Our Theorem \ref{main1} proves the isomorphism of $\kappa$ for $\mathfrak{g}$ of type ADE, and therefore provides a {\it rigid} cluster realization of ${\bf U}_q(\mathfrak{g})$ for the first time.
It is worth mentioning  that the isomorphism $\kappa$ is defined over $\mathbb{L}=\mathbb{Z}[q^{\frac{1}{2}}, q^{-\frac{1}{2}}]$. Therefore it is still valid when $q^{\frac{1}{2}}$ is a root of unity. Its connection to the small quantum groups from the perspective of cluster theory is an interesting direction for future research.

\end{itemize}

As an application of Theorem \ref{main1}, we obtain a natural $\mathbb{L}$-linear basis of ${\bf U}_q(\mathfrak{g})$ which satisfies many remarkable properties. 

\bt 
\label{main.theorem.baisi.p}
The quantum group ${\bf U}_q(\mathfrak{g})$ admits a natural linear basis $\overline{\bf \Theta}$ satisfying the following properties.
\begin{itemize}
\item[a)] The structural coefficients of $\overline{\bf \Theta}$ are in $\mathbb{N}[q^{\frac{1}{2}}, q^{-\frac{1}{2}}]$.
\item[b)] The basis $\overline{\bf \Theta}$, as a set, is preserved by Lusztig's braid group action and the Dynkin automorphisms.  
\item[c)] Every element in  $\overline{\bf \Theta}$ is self-adjoint, i.e., invariant under the star anti-involution.
\item[d)] The basis is naturally parametrized by the set ${\U}^L_+(\mathbb{Z}^t)\times X^*(\H)\times {\U}^L_+(\mathbb{Z}^t)$, where $\U^L_+(\mathbb{Z}^t)$ is the set of Lusztig data, and $X^*(\H)$ is the weight lattice of $\G$.
\end{itemize}
\et 

We shall name $\overline{\bf \Theta}$ the {\it cluster canonical bases} because they are constructed through the quantum cluster duality and our cluster model for ${\bf U}_q(\mathfrak{g})$ as in Theorem \ref{main1}. Below we include a few more remarks on related results.

\begin{itemize}
\item Lusztig introduced the algebra $\dot{\bf U}$ as a modified form of $\mathcal{U}_q(\mathfrak{g})$, and further constructed a canonical basis $\dot{\bf B}$ for $\dot{\bf U}$. Note that  $\dot{\bf U}$ is an algebra without unit. Therefore $\dot{\bf B}$ is not a basis of $\mathcal{U}_q(\mathfrak{g})$. See \cite[Part IV]{Lus} for more details. 

\smallskip 

\item Berenstein and Greenstein \cite{BG} constructed a basis ${\bf B}_{\mathfrak{g}}$, under the name of {\it double canonical basis}, for the quantum group $\mathcal{U}_q(\mathfrak{g})$, via using a variant of Lusztig's Lemma. They further established several nice properties about the 
basis. However, it is not known whether ${\bf B}_{\mathfrak{g}}$ is equivariant under Lusztig's braid group action (cf. Conjecture 1.16 of {\it loc.cit.}). The structural coefficients of ${\bf B}_{\mathfrak{g}}$ are shown to be in $\mathbb{Z}[q^{\frac{1}{2}}, q^{-\frac{1}{2}}]$, but it is not clear whether they are in $\mathbb{N}[q^{\frac{1}{2}}, q^{-\frac{1}{2}}]$.
For $\mathfrak{sl}_2$, a direct calculation shows that ${\bf B}_{\bf g}$ coincides with our basis $\overline{\bf \Theta}$. Its comparison for $\mathfrak{g}$ beyond $\mathfrak{sl}_2$ remains mysterious, and is an interesting direction for future research.

\smallskip

\item 
The study of canonical bases for (quantum) cluster algebras is a central topic in cluster theory.
The Duality Conjecture of Fock and Goncharov \cite{FGensemble} predicts that every quantum cluster algebra admits a natural basis, parametrized by the tropicalization of its Langlands dual cluster $K_2$ variety. 

\smallskip

\noindent  For the classical case when $q=1$, Gross, Hacking, Keel, and Kontsevich \cite{GHKK} constructed a family of formal power series, called $\theta$ series, by counting broken lines in the associated scattering diagrams. When the Donaldson-Thomas transformation of the cluster variety is a cluster transformation, or equivalently when its underlying quiver admits a reddening sequence in the sense of \cite{KelDT},  the $\theta$ series provide an actual linear basis of the corresponding cluster Poisson algebra. 

\smallskip 

\noindent More recently, by combining the tools developed in \cite{GHKK} and  Donaldson-Thomas theory, Davison and Mandel \cite{DavMan} prove the quantum cluster duality for skew-symmetric cases, under the same assumption on the existence of reddening sequences.

 \smallskip
 
 \item Theorem \ref{main1} allows us to apply the basis theory for quantum cluster algebras to the study of quantum groups. The proof of Theorem \ref{main.theorem.baisi.p} crucially uses several tools developed in \cite{GS2, GS3}.
First, the papers {\it loc.cit.} prove the clusterness of Donaldson-Thomas transformations for $\mathscr{P}_{\G, \bS}$. Therefore one may apply the result of \cite{DavMan} to obtain an $\mathbb{L}$-linear basis for $\mathcal{O}_q(\mathscr{P}_{\G, \bS})$. 
Second, it should be noted that our basis $\overline{\bf \Theta}$ is not exactly the theta basis of \cite{DavMan}.
Instead, in \cite{GS2, GS3}
we show that that Weyl group $W$ acts on $\mathcal{O}_q(\mathscr{P}_{\G, \odot})$ via quantum cluster automorphisms. As a consequence, one may take the sums of quantum theta functions of \cite{DavMan} along the $W$-orbits, whose projections to $\mathcal{O}_q(\mathscr{P}_{\G, \odot})^W{\Big \slash} \mathcal{I}$ give rise to the linear basis $\overline{\bf \Theta}$ for ${\bf U}_q(\mathfrak{g})$.

\end{itemize}

\medskip

\paragraph{{\bf Acknowledgments.}} I am grateful to Alexander Goncharov for urging me to finish the paper and for numerous inspiring discussions.  I wish to thank Joel Kamnitzer and Bernhard Keller for many helpful conversations. I was supported by NSF grant DMS-2200738.

\section{Preliminaries on Quantum Groups}
\label{Sec.quan.group}
\subsection{Definitions}
Let $\mathfrak{g}$ be a rank $r$ complex Lie algebra of type ADE.
Associated with $\mathfrak{g}$ is a $r\times r$ Cartan matrix $(a_{ij})$. 

\smallskip

We consider the unital associative $\mathbb{K}$-algebra $\tqg$ with the generators ${\bf E}_i, {\bf F}_i, {\bf K}_i, \widetilde{\bf K}_i$  ($1\leq i \leq r$) and  the relations
\be
\bK_i\bK_j=\bK_j\bK_i, \hskip 1cm \tbK_i\tbK_j=\tbK_j\tbK_i, \hskip 1cm \tbK_i\bK_j = \bK_j \tbK_i,
\ee
\be
\left\{ \begin{array}{cc} \bK_i\bE_j =q^{a_{ij}}\bE_j \bK_i,   \quad \quad   \tbK_i\bE_j =q^{-a_{ij}}\bE_j \tbK_i, \\
\bK_i\bF_j =q^{-a_{ij}}\bF_j \bK_i, \quad \quad \tbK_i\bF_j =q^{a_{ij}}\bF_j \tbK_i,
\end{array}
\right.
\ee
\be
\label{eq3}
\bE_i \bF_j -\bF_j \bE_i =\delta_{ij}(q - q^{-1})(\bK_i -\tbK_i),
\ee
\be
\left\{ \begin{array}{cc}
     \bE_i^2\bE_j -(q+q^{-1})\bE_i\bE_j \bE_i + \bE_j \bE_i^2=0 \quad \mbox{if } a_{ij}=-1,  \\
     \bE_i\bE_j -\bE_j\bE_i=0 \quad \mbox{if } a_{ij}=0, 
\end{array}\right.
\ee
\be
\left\{ \begin{array}{cc}
     \bF_i^2\bF_j -(q+q^{-1})\bF_i\bF_j \bF_i + \bF_j \bF_i^2=0 \quad \mbox{if } a_{ij}=-1,  \\
     \bF_i\bF_j -\bF_j\bF_i=0 \quad \mbox{if } a_{ij}=0.
\end{array}\right.
\ee
Note that $\bK_i\tbK_i$ commutes with every generator. Hence,  $\bK_i\tbK_i$ is in the center of $\tqg$. 

\smallskip

The quantum group $\uqg$ is the quotient algebra of $\tqg$ modulo the
ideal generated by $\bK_i\tbK_i-1$ for $1\leq i\leq r$. The generators $\{\bE_i, \bF_i, \bK_i\}$ are the rescaled version of the generators $\{E_i, F_i, K_i\}$ used in \cite{Lus}:
\be
\la{scaled.relation}
\bE_i= q^{-\frac{1}{2}}(q-q^{-1})E_i, \qquad \bF_i=q^{\frac{1}{2}}(q^{-1}-q)F_i, \qquad {\bK}_i = K_i.
\ee

\smallskip

We denote by 
\[\mathcal{D}_q(\mathfrak{b}):=\tqg[\bK_i^{-1}, \tbK_i^{-1}]_{1\leq i\leq r}\] 
the localization of $\tqg$ at the Cartan part. The algebra $\mathcal{D}_q(\mathfrak{b})$ coincides with the quantized Drinfeld double of the Borel subalgebra $\mathfrak{b}$.

\subsection{Braid group action and PBW basis}

For $1\leq i \leq r$, let $T_i$ be the $\mathbb{K}$-algebra automorphism of $\mathcal{D}_q(\mathfrak{b})$ defined by 
\[
\bE_j \mapsto q^{-1}\bK_j^{-1}\bF_j, \quad \bF_j \mapsto q\bE_j\tbK_j^{-1}, \quad  \bK_j \mapsto \bK_j^{-1}, \quad  \tbK_j \mapsto \tbK_j^{-1}, \qquad   \mbox{if }j=i; \]
\[   
\bE_j \mapsto \bE_j, \quad  \bF_j \mapsto \bF_j, \quad  \bK_j\mapsto \bK_j, \quad  \tbK_j \mapsto \tbK_j, \qquad   \mbox{if } a_{ij}=0; \]
\[
\left\{ \begin{array}{c}\displaystyle{\bE_j \mapsto \frac{q^{1/2}\bE_j\bE_i - q^{-1/2}\bE_i\bE_j}{q-q^{-1}}, \quad  \bF_j \mapsto \frac{q^{1/2}\bF_j\bF_i - q^{-1/2}\bF_i\bF_j}{q-q^{-1}}},\\  \bK_j \mapsto \bK_i\bK_j, \quad   \tbK_j \mapsto \tbK_i\tbK_j,\end{array} \right.  \qquad \mbox{if } a_{ij}=-1.\]
The automorphisms $T_i$ satisfy the braid relations:
\[
\left\{ \begin{array}{cl}
    T_iT_jT_i=T_jT_iT_j & \mbox{if } a_{ij}=-1,  \\
    T_iT_j = T_j T_i & \mbox{if } a_{ij}=0. 
\end{array}
\right.
\]
Therefore, they give rise to a braid group action on $\mathcal{D}_q(\mathfrak{b})$. 

Note that $\widetilde{\mathcal{U}}_q(\mathfrak{g})$ is not closed under the braid group action. In particular, \[T_i({\bf K}_i)={\bf K}_i^{-1}\notin \widetilde{\mathcal{U}}_q(\mathfrak{g}).\]

Each $T_i$ preserves the ideal generated by all $\bK_j\tbK_j-1$.
Therefore, the braid group action on $\mathcal{D}_q(\mathfrak{b})$ descends to an action on $\uqg$. After the normalization \eqref{scaled.relation}, this braid group action coincides with Lusztig's braid group action on $\uqg$. 

\smallskip

Fix a reduced word ${\bf i}=(i_1, \ldots, i_n)$ for the longest element $w_0$ in the Weyl group associated with $\mathfrak{g}$. For $1\leq k\leq n$, we set 
\be
\label{pbw.c1}
\begin{array}{c}
    \bE_{{\bf i}, k}:= T_{i_1}T_{i_2}\ldots T_{i_{k-1}} \bE_{i_k},   \\
       \bF_{{\bf i}, k}:= T_{i_1}T_{i_2}\ldots T_{i_{k-1}} \bF_{i_k}.
\end{array}
\ee
Let $\mathbb{N}=\{0,1,2, \ldots\}$. Given two vectors $\vec{a}=(a_1, \ldots, a_n)$ and $\vec{c}=(c_1, \ldots, c_n)$ in $\mathbb{N}^{n}$, set 
\[
\bE_{\bf i}(\vec{a})= \bE_{{\bf i},1}^{a_1} \bE_{{\bf i},2}^{a_2}\cdots \bE_{{\bf i},n}^{a_n},
\quad \quad 
\bF_{\bf i}(\vec{c})= \bF_{{\bf i},1}^{c_1} \bF_{{\bf i},2}^{c_2}\cdots \bF_{{\bf i},n}^{c_n}.
\]
For $\vec{b}=(b_1,\ldots,b_r)$ and $\vec{d}=(d_1, \ldots, d_r)$ in $\mathbb{Z}^r$, we set
\be
\label{PBW.as12}
\bK(\vec{b})= \bK_{1}^{b_1}\cdots \bK_{r}^{b_r},
\quad \quad 
\tbK(\vec{d})= \tbK_{1}^{d_1}\cdots \tbK_{r}^{d_r}.
\ee
The set 
\be
\label{dimofp}
\left 
\{
\bE_{\bf i}(\vec{a})\cdot \bK(\vec{b}) \cdot \bF_{\bf i}(\vec{c})\cdot \tbK(\vec{d}) ~\middle |~ (\vec{a}, \vec{b}, \vec{c},\vec{d}) \in \mathbb{N}^{2n+2r}
\right\}
\ee
forms a $\mathbb{K}$-linear basis of $\tqg$, called the Poincare-Birkhoff-Weil (PBW) basis. 
We denote by $\tbUq$ the $\mathbb{L}$-linear span of the PBW basis. Following Theorem \ref{iso,mai}, we see that $\tbUq$ is a $\mathbb{L}$-subalgebra inside $\tqg$, and does not depend on the reduced word ${\bf i}$ chosen. 

\smallskip 

Similarly, the algebra $\mathcal{D}_q(\mathfrak{b})$ has a $\mathbb{K}$-linear basis
\[
\left 
\{
\bE_{\bf i}(\vec{a})\cdot \bK(\vec{b}) \cdot \bF_{\bf i}(\vec{c})\cdot \tbK(\vec{d}) ~\middle |~ \vec{a}, \vec{c}\in \mathbb{N}^n;  \vec{b},\vec{d} \in \mathbb{Z}^{r}
\right\}.
\]
Let ${\bf D}_q(\mathfrak{b})$ be the $\mathbb{L}$-linear span of the above basis. Then ${\bf D}_q(\mathfrak{b})$ is a  $\mathbb{L}$-subalgebra of $\mathcal{D}_q(\mathfrak{b})$.

\smallskip

By imposing the conditions $\bK_i\tbK_i=1$, we obtain the quantum group $\Uq(\mathfrak{g})$ with a basis
\[
\left 
\{
\bE_{\bf i}(\vec{a})\cdot \bK(\vec{b}) \cdot \bF_{\bf i}(\vec{c}) ~\middle |~ \vec{a}, \vec{c}\in \mathbb{N}^n;  \vec{b} \in \mathbb{Z}^{r}
\right\}.
\]
Their $\mathbb{L}$-linear gives a $\mathbb{L}$-subalgebra $\bUq$ inside $\Uq(\mathfrak{g})$.

\subsection{Dynkin automorphisms and stat anti-involution}
A Dynkin automorphism $\sigma$ of $\mathfrak{g}$ is a permutation of the vertices of the Dynkin diagram of $\mathfrak{g}$ that preserves the Cartan matrix: 
\[
a_{\sigma(i)\sigma(j)}= a_{ij}.
\]
The set of Dynkin automorphisms forms a finite group and coincides with the outer automorphism group ${\rm Out}(\G)$ of $\G$.

Each $\sigma$ includes an algebra automorphism of ${\bf D}_q(\mathfrak{b})$ such that 
\[
{\bf E}_i \longmapsto {\bf E}_{\sigma(i)}; \qquad {\bf F}_i \longmapsto {\bf F}_{\sigma(i)}; \qquad {\bf K}_i\longmapsto {\bf K}_{\sigma(i)}; \qquad \widetilde{\bf K}_i\longmapsto \widetilde{\bf K}_{\sigma(i)}.
\]
The action $\sigma$ further induces an action on $\widetilde{\bf U}_q(\mathfrak{g})$ and on ${\bf U}_q(\mathfrak{g})$.

\smallskip
The {\it star anti-involution} $\ast$ is an anti-automorphism of ${\bf D}_q(\mathfrak{b})$ that preserves the generators ${\bf E}_i, {\bf F}_i, {\bf K}_i, \widetilde{\bf K}_i$, maps $q$ to $q^{-1}$, and satisfies
\[
({\bf fg})^\ast ={\bf g}^\ast{\bf f}^\ast, \qquad \forall {\bf f},{\bf g}\in {\bf D}_q(\mathfrak{b}).  
\]
The map $\ast$ makes ${\bf D}_q(\mathfrak{b})$ a $\ast$-algebra. 
We say an element ${\bf f}$ in ${\bf D}_q(\mathfrak{b})$ is {\it self-adjoint} if ${\bf f}={\bf f}^\ast$. 
The same anti-involution $\ast$ descends to $\widetilde{\U}_q(\mathfrak{g})$ and ${\U}_q(\mathfrak{g})$.

\section{Basics on Cluster Algebras}
For the convenience of the reader, we briefly recall several basic definitions concerning the quantum cluster algebra of Fock and Goncharov, mainly following the conventions of \cite{FGensemble}. 
It is worth mentioning that the Fock-Goncharov quantum cluster algebras are different from the Berenstein-Zelevinsky version of quantum cluster algebras in \cite{BZqua}. In particular, the latter is well defined only when the underlying exchange matrices are of full rank, and its definition depends on a particular choice of an inverse matrix. The Fock-Goncharov quantum cluster algebra is well defined for arbitrary exchange matrices. Further relations about the above two quantum cluster algebras can be found in the appendix of \cite{GS3}.

\subsection{Quiver mutations}
An {\it ice quiver} a triple $Q=(V, V_0, \varepsilon)$, where $V$ is a finite set, $V_0$ is a subset of $V$, and $\varepsilon=\{\varepsilon_{ij}\}$ is a $\frac{1}{2}\mathbb{Z}$ valued function on $V\times V$, such that $\varepsilon_{ij}=-\varepsilon_{ji}$, and $\varepsilon_{ij}$ is an integer unless $i, j \in V_0$.
Such an ice quiver $Q$ can be visualized by a directed graph without loops or two cycles, such that $V$ is the set of vertices, and the number of arrows from $i$ to $j$ is 
$[\varepsilon_{ij}]_+:= \max\{0,\varepsilon_{ij}\}$. Elements in $V_0$ are called frozen vertices and are usually denoted by boxes in their visualization. 

\smallskip

Let $k\in V-V_0$. The {\it quiver mutation} in the direction $k$ produces a new ice quiver $u_kQ=(V, V_0, \varepsilon')$ such that 
\begin{align}
\label{quiver mutation}
\varepsilon_{ij}'&= \begin{cases}
-\varepsilon_{ij}, &\mbox{if } k\in \{i,j\},\\
\varepsilon_{ij}, &\mbox{if } \varepsilon_{ik}\varepsilon_{kj}<0, ~~ k\notin\{i,j\},\\
\varepsilon_{ij}+|\varepsilon_{ik}|\varepsilon_{kj}, &\mbox{if } \varepsilon_{ik}\varepsilon_{kj}>0, ~~ k\notin\{i,j\}.
\end{cases}
\end{align}
Note that the quiver mutation is involutive: $\mu_k^2(Q)=Q$.
\begin{example}
Let $V=\{a,b,c,d\}$, $V_0=\{c,d\}$, and 
\[\varepsilon=
\begin{pmatrix}
0 & -1 & 0 & 0 \\
1 & 0 & -1 & 1 \\
0 & 1 & 0 & -\frac{1}{2}\\
0 & -1 & \frac{1}{2} & 0 \\
\end{pmatrix}
\]
The ice quiver $Q=(V, V_0, \varepsilon)$ is illustrated by the left graph in Figure \ref{fig1}. The dashed arrow between the frozen vertices  $c$ and $d$ indicates that $\varepsilon_{dc}=\frac{1}{2}$. The quiver mutation in the direction $b$ produces an ice quiver as illustrated by the right graph.  
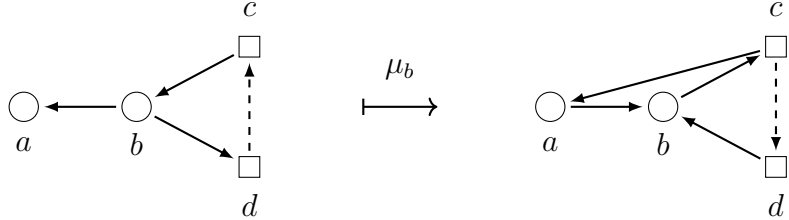
\begin{figure}[ht]
\begin{tikzpicture}
\begin{scope}[>=latex]
 \node[circle,draw] (C) at (0,0) {};
 \node at (0,-0.5) {$a$};
 \node[circle,draw] (D) at (1.5,0) {};
 \node at (1.5,-0.5) {$b$};
 \node[rectangle,draw] (A) at (3,0.8) {};
 \node at (3,1.3) {$c$};
 \node[rectangle,draw] (B) at (3,-0.8) {};
 \node at (3, -1.3) {$d$};
\qarrow{D}{C};
\qarrow{A}{D};
\qarrow{D}{B};
\qdarrow{B}{A};
\end{scope}
\draw[thick,|->] (4.5,0) -- (5.5,0);
\node at (5,0.5) {$\mu_b$};
\begin{scope}[>=latex, xshift=7cm]
 \node[circle,draw] (C) at (0,0) {};
 \node at (0,-0.5) {$a$};
 \node[circle,draw] (D) at (1.5,0) {};
 \node at (1.5,-0.5) {$b$};
 \node[rectangle,draw] (A) at (3,0.8) {};
 \node at (3,1.3) {$c$};
 \node[rectangle,draw] (B) at (3,-0.8) {};
 \node at (3, -1.3) {$d$};
\qarrow{C}{D};
\qarrow{D}{A};
\qarrow{A}{C};
\qarrow{B}{D};
\qdarrow{A}{B};
\end{scope}
\end{tikzpicture}
\caption{A quiver mutation in the direction $b$.}
\label{fig1}
\end{figure}
\end{example}

\subsection{Fock-Goncharov quantum cluster algebras}
Let $Q$ be an ice quiver. The quantum torus algebra ${\bf T}_Q$ is a $\mathbb{L}$-algebra generated by $X_{i}^{\pm 1} ~(i\in V)$, satisfying 
\[
X_iX_j =q^{2\varepsilon_{ij}} X_jX_i.
\]
The pair $\Sigma=(Q, {\bf T}_{Q})$ is called a {\it quantum seed}. 

Let $\Lambda$ be a free abelian group with basis $\{e_i\}_{i\in V}$. Let us choose an order $e_1, \ldots, e_m$ of the basis of $\Lambda$, where $m$ is the cardinality of $V$. Following the notation of \cite[\S 3.1]{FGensemble}, for each vector $a=\sum_{i\in V}a_ie_i \in \Lambda$, we introduce the normalized monomial 
\be
\la{norm.monomial}
X_a:= q^{-\sum_{i<j}a_ia_j \varepsilon_{ij}}\prod_{i=1}^m X_i^{a_i}
\ee
for future use. Note that the monomial $X_a$ does not depend on the order chosen. The set $\{X_a\}_{a\in \Lambda}$ forms a $\mathbb{L}$-linear basis of ${\bf T}_Q$.

Let $\mathcal{F}:={\rm Frac}({\bf T}_Q)$ be the non-commutative fraction field of the quantum torus algebra ${\bf T}_Q$. For example, see the Appendix of \cite{BZqua} for more details on the construction of $\mathcal{F}$. A mutation in the direction $k\in V-V_0$ creates a new collections of variables  $\{X_i'\}_{i\in V}$, where
\be
\label{Q.transition}
X_i'=\left\{ \begin{array}{ll}
X_k^{-1}, &\mbox{if } i=k,\\
X_i\displaystyle{\prod_{r=1}^{|\varepsilon_{ik}|} (1+q^{2r-1}X_k)}, &\mbox{if } \varepsilon_{ik}<0,\\
X_i\displaystyle{\prod_{r=1}^{\varepsilon_{ik}}(1+q^{1-2r}X_k^{-1})},
 &\mbox{if } \varepsilon_{ik}\geq 0.
\end{array}\right.
\ee
The new variables $X_i'$ satisfy the relation  \[X_i'X_j'=q^{2\varepsilon_{ij}'}X_j'X_i',\] where $\varepsilon_{ij}'$ is given by \eqref{quiver mutation}. 

The quantum torus algebra ${\bf T}_{\mu_k(Q)}$ is the $\mathbb{L}$-subalgebra of $\mathcal{F}$ generated by $X_i'^{\pm 1}$, $i\in V$.
The pair $\mu_k(\Sigma)=(\mu_kQ, {\bf T}_{\mu_k Q})$ is a mutated quantum seed in the direction $k$. We say that a quantum seed $\Sigma'=(Q', \mathbf{T}_{Q'})$ is mutation equivalent to $\Sigma$ if it can be obtained from $\Sigma$ by a sequence of quantum seed mutations. Denote by $|\Sigma|$ the class of quantum seeds that are mutation equivalent to $\Sigma$. 
\begin{definition}
Let $\Sigma=(Q, {\bf T}_Q)$ be a quantum seed. The quantum cluster algebra $\mathcal{O}_q(\mathscr{X}_{|\Sigma|})$ is the intersection of the quantum torus algebras $\mathbf{T}_{Q'}$ for all quantum seeds $(Q', \mathbf{T}_{Q'})\in |\Sigma|$:
\[
\mathcal{O}_q(\mathscr{X}_{|\Sigma|}):= \bigcap_{(Q', {\bf T}_{Q'})\in |\Sigma|} {\bf T}_{Q'} \subset \mathcal{F}.
\]
We will frequently write $\mathcal{O}_q(\mathscr{X}_{|\Sigma|})$ as $\mathcal{O}_q(\mathscr{X})$ when there is no ambiguity. 
\end{definition}

\begin{remark} The algebra $\mathcal{O}_q(\mathscr{X}_{|\Sigma|})$ has a natural $\ast$-algebra structure. Namely, there is an anti-automorphism $\ast$ on $\mathcal{O}_q(\mathscr{X}_{|\Sigma|})$ that preserves all the generators $X_i$, maps $q$ to $q^{-1}$. See \cite{FGrep} on a beautiful construction on the $\ast$-representations of $\mathcal{O}_q(\mathscr{X}_{|\Sigma|})$.
\end{remark}

\begin{remark}
The semiclassical limit of $\mathcal{O}_q(\mathscr{X}_{|\Sigma|})$ as $q^{\frac{1}{2}}\rightarrow 1$ gives rise to a cluster Poisson algebra $\mathcal{O}(\mathscr{X}_{|\Sigma|})$.
Geometrically, each seed in $|\Sigma|$ corresponds to an algebraic torus $\left(\mathbb{G}_m\right)^{\# V}$. We may glue all the tori together via the classical version of the transition maps \eqref{Q.transition}, obtaining a cluster Poisson variety $\mathscr{X}_{|\Sigma|}$.  
\end{remark}

For a quantum seed $\Sigma'=(Q', {\bf T}_{Q'})\in |\Sigma|$,  suppose $a=\sum_{i\in V} a_i e_i'$ is a vector such that 
\be
\sum_{i\in V}a_i \varepsilon_{ij}'\geq 0,\qquad \forall j\in V-V_0. 
\ee
The monomial $X_a'$, defined as in \eqref{norm.monomial}, is a global function, i.e., $X_a'\in \mathcal{O}_q(\mathscr{X}_{|\Sigma|})$. In this case, we call $X_a'$ a {\it global monomial}. 

When $\sum_{i\in V}a_i \varepsilon_{ij}'=0$ for every $j\in V$, then the global monomial $X_a'$ belongs to the center of $\mathcal{O}_q(\mathscr{X}_{|\Sigma|})$. Furthermore, for any $n\in \mathbb{Z}$, the power $X_{na}'=(X_{a}')^n$ remains a monomial in every seed in $|\Sigma|$. Such a element $X_a'$ is called a {\it Casimir}.

\subsection{Cluster $K_2$ varieties and tropicalization}

Let $Q=(V, V_0, \varepsilon)$ be an ice quiver. Associated with $Q$ is a split algebraic torus $\mathcal{A}_Q=(\mathbb{G}_m)^{\# V}$ with coordinates $\{A_i\}_{i\in V}$. The pair $\Sigma=(Q, \mathcal{A}_Q)$ is called a {\it cluster $K_2$ seed}. 

\smallskip 

The cluster $K_2$ mutation in the direction $k\in V-V_0$ creates a new seed $(\mu_kQ, \mathcal{A}_{\mu_{kQ}})$, where $\mathcal{A}_{\mu_kQ}$ is an algebraic torus with coordinates $\{A_i'\}_{i\in V}$ such that
\be
\la{k2mutation}
A_i'=\left\{ \begin{array}{cl}
A_i  &\mbox{if $i\neq k$};\\ 
{\displaystyle{\frac{\prod_{\varepsilon_{ki}\geq 0} A_i^{\varepsilon_{ki}}+\prod_{\varepsilon_{ki}\leq 0} A_i^{-\varepsilon_{ki}}}{A_k}}} & \mbox{if $i=k$}.
\end{array}
\right.
\ee
We recursively repeat the same procedure in the above new seeds in all the possible directions in $V-V_0$, obtaining possibly infinitely many cluster $K_2$ seeds. 
The {\it cluster $K_2$ variety} $\mathscr{A}_Q$ is the union of all the obtained cluster tori $\mathcal{A}_{Q'}$, glued together via sequences of the transition maps  \eqref{k2mutation}:
\[
\mathscr{A}_{|\Sigma|}=\bigcup \mathcal{A}_{Q'}.
\]
By definition, the ring of regular functions on $\mathscr{A}_{|\Sigma|}$ coincides with the upper cluster algebra of Bereinstein, Fomin, and Zelevinsky \cite{BFZ}.   The set $\{A_i'\}_{i\in V}$ of cluster $K_2$ variables for each $\mathcal{A}_{Q'}$ is called a {\it cluster chart}. The well-known Laurent phenomenon of cluster algebras asserts that
all the cluster variables $A_i'$ are regular functions on $\mathscr{A}_{|\Sigma|}$.
The variables $A_i$ for $i\in V_0$ are invariant under mutations and therefore are called {\it frozen cluster variables}.

\smallskip

The transition map \eqref{k2mutation} is subtraction free, and therefore gives rise to a positive structure  on $\mathscr{A}_{|\Sigma|}$. Recall that a nonzero rational function $F$ on $\mathscr{A}_{|\Sigma|}$ is positive if it can be presented as a ratio of polynomials with non-negative integral coefficients in the cluster variables of one (and hence every) cluster chart. Denote by $\mathbb{Q}_+(\mathscr{A}_{|\Sigma|})$ the set of all positive functions. The set $\mathbb{Q}_+(\mathscr{A}_{|\Sigma|})$ forms a semifield, i.e., it is closed under addition, multiplication, and addition. 

The tropical semifield $\mathbb{Z}^t$ is a set of integers with the multiplication $\cdot_t$ and the addition $+_t$ given by 
\[
a\cdot_tb = a+b, \qquad a+_tb =\min\{a, b\}.
\]
The  tropicalization of $\mathscr{A}_{|\Sigma|}$ is the set of semifield homomorphisms from $\mathbb{Q}_+(\mathscr{A}_{|\Sigma|})$ to $\mathbb{Z}^t$:
\[
\mathscr{A}_{|\Sigma|}(\mathbb{Z}^t):= {\rm Hom}_{\mbox{\it semifield}}\left(\mathbb{Q}_+(\mathscr{A}_{|\Sigma|}), \mathbb{Z}^t\right).
\]
Tautologically, for any positive function $F\in \mathbb{Q}_+(\mathscr{A}_{|\Sigma|})$, its tropicalization is a $\mathbb{Z}$-valued function on $\mathscr{A}_{|\Sigma|}(\mathbb{Z}^t)$, defined as
\[
F^t(l):=l(F), \qquad \forall l \in \mathscr{A}_{|\Sigma|}(\mathbb{Z}^t).
\]

\smallskip
The variety $\mathscr{A}_{|\Sigma|}$ is equipped with a canonical closed 2-form $\Omega$,  which be expressed in a cluster chart as 
\[
\Omega= \sum \varepsilon_{ij}\frac{A_i}{A_i}\wedge \frac{A_j}{A_j}.
\]
Following \cite{Fra}, a quasi-cluster automorphism of $\mathscr{A}_{|\Sigma|}$ is an automorphism of $\mathscr{A}_{|\Sigma|}$ that preserves the 2-form $\Omega$ and can be presented by a sequence of cluster mutations as \eqref{k2mutation} followed by a renormalization by frozen cluster variables. Denote by $\mathcal{G}_{|\Sigma|}$ the {\it quasi cluster modular group} that consists of all the quasicluster automorphisms.  The group $\mathcal{G}_{|\Sigma|}$ also acts on the quantum cluster algebra $\mathcal{O}_q(\mathscr{X}_{|\Sigma|})$. See \cite[\S19.2]{GS3} for more details.

\subsection{Quantum cluster duality} The quantum cluster Duality Conjecture, suggested by Fock and Goncharov in \cite[\S 4.3]{FGensemble}, asserts that the algebra $\mathcal{O}_q(\mathscr{X})$ admits a natural $\mathbb{L}$-linear basis, which can be canonically parametrized by the tropicalization of the Langlangds dual\footnote{In the skew-symmetric cases, which are the main setting of this paper, the Langlands dual $\mathscr{A}^L$ coincides with $\mathscr{A}$.} cluster $K_2$ variety $\mathscr{A}^L$. When the Donaldson-Thomas transformation of the cluster variety $\mathscr{X}_{|{\Sigma}|}$ is a cluster transformation, or equivalently when the underlying ice quiver $Q$ of $\mathscr{X}_{|{\Sigma}|}$ admits a reddening sequence in the sense of \cite{KelDT}, the classical version of the conjecture has been proven by Gross, Hacking, Keel, and Kontsevich in \cite{GHKK}. More recently, Davison and Mandel \cite{DavMan} prove the quantum cluster duality for skew-symmetric cases, under the same assumption on the existence of reddening sequences. Since the present paper will crucially use the quantum cluster duality, we recall its precise statement in this subsection.

\smallskip

Every seed $(Q', \{A_i'\})$ of $\mathscr{A}_{|\Sigma|}$ corresponds to a quantum seed $(Q', {\bf T}_{Q'})$ of $\mathcal{O}_q(\mathscr{X}_{|\Sigma|})$. Every tropical point $l\in \mathscr{A}_{|\Sigma|}(\Z^t)$ corresponds to a vector
\[
a(l):=\sum l(A_i') e_i', 
\]
which further corresponds to a monomial $X_{a(l)}'$ as in \eqref{norm.monomial}. Let $b=\sum b_i e_i'$. We introduce the partial order $b\geq a(l)$, if $b_i \geq l(A_i')$ for every $i\in V-V_0$ and $b_i=l(A_i')$ for every $i\in V_0$.

\smallskip

Now we are ready to state the quantum cluster duality result, proven by Davison and Mandel in \cite{DavMan}.
\bt
\label{quantum.cluster.dual}
Assume that the unfrozen part of the underlying quiver $Q$ admits a reddening sequence.
The algebra $\mathcal{O}_q(\mathscr{X}_{|\Sigma|})$ admits a $\mathbb{L}$-linear basis $\Theta\left(\mathcal{O}_q(\mathscr{X}_{|\Sigma|})\right)$, called the {\it quantum theta basis}. 

The basis $\Theta\left(\mathcal{O}_q(\mathscr{X}_{|\Sigma|})\right)$ satisfies the following properties.
\begin{itemize}
\item[1)] The basis $\Theta\left(\mathcal{O}_q(\mathscr{X}_{|\Sigma|})\right)$  is preserved by the action of $\mathcal{G}_{|\Sigma|}$.
\item[2)] All the global monomials are contained in $\Theta\left(\mathcal{O}_q(\mathscr{X}_{|\Sigma|})\right)$.
\item[3)] The structure constants for the multiplication of elements in $\Theta\left(\mathcal{O}_q(\mathscr{X}_{|\Sigma|})\right)$ are in $\mathbb{N}[q^{\frac{1}{2}},q^{-\frac{1}{2}}]$. 
\item[4)] Every quantum theta function in $\Theta(\mathcal{O}_q(\mathscr{X}_{|\Sigma|}))$ is self-adjoint, i.e., invariant under the anti-involution $\ast$ of $\mathcal{O}_q(\mathscr{X}_{|\Sigma|})$.
\item[5)] There is natural $\mathcal{G}_{|\Sigma|}$-equivarient bijection
\be
\label{dual.i}
\mathscr{A}_{|\Sigma|}(\mathbb{Z}^t)\stackrel{\sim}{\longrightarrow}\Theta\left(\mathcal{O}_q(\mathscr{X}_{|\Sigma|})\right), \qquad l \longmapsto \theta_l.
\ee
\item[6)] Let $(Q', \{A_i'\})$ be an arbitrary seed of $\mathscr{A}_{|\Sigma|}$. For every $l\in 
\mathscr{A}_{|\Sigma|}(\mathbb{Z}^t)$, we have
\be
\la{sann}
\theta_l = X'_{a(l)} + \sum_{v>a(l)} c_{l,v} X_v',
\ee
where $c_{l,v}\in \mathbb{N}[q^{\frac{1}{2}},q^{-\frac{1}{2}}]$.
\end{itemize}
\et

In the rest of the Section, we explore further properties of the quantum duality for future use. 
By 3) of Theorem \ref{quantum.cluster.dual}, for any  ${l_1}, {l_2} \in \mathscr{A}_{|\Sigma|}(\mathbb{Z}^t)$, we have a finite linear decomposition
\[
\theta_{l_1} \cdot \theta_{l_2} = \sum_{l\in \mathscr{A}_{|\Sigma|}(\mathbb{Z}^t)} c_q(l_1,l_2;l) \theta_{l},
\]
where  $c_q(l_1,l_2;l)\in \mathbb{N}[q^{\frac{1}{2}}, q^{-\frac{1}{2}}]$. 
Following the construction of \cite{Sh}, we define the support 
\be
\la{adcx}
{\rm Supp}(l_1, l_2):=\left\{l\in \mathscr{A}_{|\Sigma|}(\mathbb{Z}^t) ~\middle|~ c_q(l_1,l_2;l)\neq 0\right\}. 
\ee

\bl
\la{support.lemma}
If $l\in {\rm Supp}(l_1, l_2)$, then for any cluster $K_2$ variable $A$ of $\mathscr{A}_{|\Sigma|}$ we have
\be
\la{support.lemma,a}
A^t(l) \geq A^t(l_1) + A^t(l_2).
\ee
In particular, when $A$ is frozen, then the formula \eqref{support.lemma,a} achieves equality. 
\el
\begin{proof} Let us fix a seed whose  cluster chart contains the variable $A$. By  \eqref{sann}, we have
\[
\theta_{l_1}\cdot \theta_{l_2} =\left(X_{a(l_1)}'+\sum_{v>a(l_1)} c_{l_1, v}X_v' \right)\left(X_{a(l_2)}'+\sum_{u>a(l_2)} c_{l_2, u}X_v' \right)= \sum_{l\in \mathscr{A}_{|\Sigma|}(\mathbb{Z}^t)} c_q(l_1,l_2;l) \theta_{l}.
\]
Here $c_{l_1,v}$, $c_{l_2, u}$, $c_q(l_1,l_2;l)$ are in $\mathbb{N}[q^{\frac{1}{2}},q^{-\frac{1}{2}}]$. By comparing them, we get $a(l) \geq a(l_1)+ a(l_2)$, and the equality holds when $A$ is frozen. The Lemma follows.
\end{proof}

\begin{lemma}
Let $A_i$ be a frozen cluster variable. There is a unique $l_i\in \mathscr{A}_{\Sigma}(\mathbb{Z}^t)$ such that $l_i(A_i)=1$ and $l_i(A)=0$ for all the other cluster $K_2$ variables. Through the bijection \eqref{dual.i}, every frozen vertex $i\in V_0$ gives rise to a quantized theta function
\be
\label{frozen-theta-function}
\theta_{l_i} \in \Theta\left(\mathcal{O}_q(\mathscr{X}_{|\Sigma|})\right).
\ee
\end{lemma}
\begin{proof} Let us start with an arbitrary  cluster chart $\{A_j\}$. There is a unique tropical point $l_i$ such that $l_i(A_j)=\delta_{ij}$.  Let us mutate the seed in the direction $k$. By the tropicalization of \eqref{k2mutation}, we get
\[
l_i(A_k')= \min\left\{\sum_{j|\varepsilon_{kj}\geq 0} \varepsilon_{kj}l_i(A_j), \sum_{j~|~\varepsilon_{kj}< 0}-\varepsilon_{kj}l_i(A_j)\right\} - l_i(A_k).
\]
Note that $l_i(A_i)=1$ appears at most in one sum inside $\min$ and $l_i(A_k)=0$. Therefore $l_i(A_k')=0$. By induction, we show that $l_i(A)=0$ for all the cluster variables $A\neq A_i$.
\end{proof}

\smallskip

Let $\alpha=\{a_i\}_{i\in V_0}$ be a tuple of integers indexed by the frozen vertices. Define the subset of tropical points
\[
\mathscr{A}_{|\Sigma|}(\mathbb{Z}^t)_{\alpha}:=\left\{ l\in \mathscr{A}_{|\Sigma|}(\mathbb{Z})~\middle | ~ A_i^t(l)=a_i, ~\forall i\in V_0\right\}.
\]
Correspondingly, we consider the $\mathbb{L}$-linear span
\[
\mathcal{O}_q(\mathscr{X}_{|\Sigma|})_{\alpha}:= \bigoplus_{l\in \mathscr{A}_{|\Sigma|}(\mathbb{Z}^t)_{\alpha}} \mathbb{L} \theta_l.
\]

\bl \label{decomp.quant.cl}
The decomposition
\be
\mathcal{O}_q(\mathscr{X}_{|\Sigma|}):= \bigoplus_{\alpha \in \mathbb{Z}^{\# V_0}} \mathcal{O}_q(\mathscr{X}_{|\Sigma|})_{\alpha},
\ee
makes the algebra $\mathcal{O}_q(\mathscr{X}_{|\Sigma|})$ a $\mathbb{Z}^{\# V_0}$ graded algebra.
\el
\begin{proof} It follows directly from Lemma \ref{support.lemma} when $A_i$ is frozen. 
\end{proof}

Let ${\bf Ca}$ be a subset of $\mathscr{A}_{|\Sigma|}(\mathbb{Z}^t)$  consisting of  tropical points $c$ whose tropical coordinates $\{A_i^t(c)\}$ for one (and therefore every) cluster chart $\{A_i\}$ satisfy 
\[
\sum_{i\in V}A_i^t(c)\varepsilon_{ij}=0, \qquad \forall j\in V.
\]
Note that every $c\in {\bf Ca}$ correponds to a Casimir element $\theta_c$.

Let $c\in {\bf Ca}$ and let $l\in \mathscr{A}_{|\Sigma|}(\mathbb{Z}^t)$. An easy recursive check shows that there is a unique tropical point, denoted by $l+c$, such that 
\[
A^t(l+c)=A^t(l)+A^t(c)
\]
for every cluster variable $A$.
\bl 
\label{lemma.case12}
We have
\[
\theta_{l+c}=\theta_l\cdot \theta_c, \qquad \forall l\in \mathscr{A}_{|\Sigma|}(\mathbb{Z}^t), ~\forall c\in {\bf Ca}.
\]
\el
\begin{proof} By 3) of Theorem \ref{quantum.cluster.dual}, we have
\be
\label{eq1.35}
\theta_{l}\cdot \theta_c =\sum_{t} c_q(l,c;t) \theta_t,
\ee
where $c_q(l,c;t)\in \mathbb{N}[q^{\frac{1}{2}}, q^{-\frac{1}{2}}]$. In particular, by a direct comparison of the leading coefficients, we have $c_q(l,c;l+c)=1$. Note that $\theta_{-c}=\theta_c^{-1}$ is a Casimir. Therefore, 
\[
\theta_l= \sum_t c_q(l,c;t)\theta_t\theta_{-c}=\sum_{t,s}c_q(l,c;t)c_q(t,-c;s)\theta_s.
\]
Due to the non-negativity of the coefficients, we see that there is only one nontrivial term in \eqref{eq1.35}. Hence $\theta_l\cdot \theta_c=\theta_{l+c}$.
\end{proof}
Let $C=\{c_1,\ldots, c_r\} \in {\bf Ca}$. We define an $\mathbb{Z}^r$ action on the set $\mathscr{A}_{|\Sigma|}(\Z^t)$ such that
\be
\label{ascjna}
r_a(l):= l+\sum_{i=1}^r a_ic_i, \qquad \forall l\in \mathscr{A}_{|\Sigma|}(\mathbb{Z}^t),~ \forall a=(a_1,\ldots,a_r) \in \mathbb{Z}^r.
\ee
Denote by ${\rm Or}_{C}(\mathscr{A}_{|\Sigma|}(\mathbb{Z}^t))$ the set of $\mathbb{Z}^r$-orbits inside $\mathscr{A}_{|\Sigma|}(\mathbb{Z}^t)$ under the above action. Let $\mathcal{I}_C$ be an ideal of $\mathcal{O}_q(\mathscr{X}_{|\Sigma|})(\mathbb{Z}^t)$ generated by 
\[
\theta_{c_i}-1, \qquad \forall c_i\in C.
\]

The following proposition is a direct consequence of Lemma \ref{lemma.case12}.
\begin{proposition}
The quantum theta basis of $\mathscr{O}_q(\mathscr{X}_{|\Sigma|})$ descends to a natural $\mathbb{L}$-linear basis of the quotient algebra \[\mathscr{O}_q(\mathscr{X}_{|\Sigma|}){\big \slash} \mathcal{I}_C,\] 
naturally parametrized by the set ${\rm Or}_{C}(\mathscr{A}_{|\Sigma|}(\mathbb{Z}^t))$.
\end{proposition}

\section{Decorated character variety}
Let $\G$ be a split semi-simple algebraic group over $\mathbb{Q}$ with trivial center. Let $\tilde{\G}$ be the universal cover of $\G$. A decorated surface $\bS$ is an oriented topological surface with {\it punctures} inside, and a finite number of {\it marked points} on its boundary.  The moduli space of $\G$-local systems over $\bS$ (a.k.a.  the {\it character variety}) is
\[\mathscr{L}_{\G,\bS}={\bf Hom}\left(\pi_1(\bS),\G\right)\Big/\G.\]

In their seminal work \cite{FGteich}, Fock and Goncharov introduce a pair  $(\mathscr{X}_{\G, \bS}, \mathscr{A}_{\tilde{\G}, \bS})$ of moduli spaces, which are varients of $\mathscr{L}_{\G,\bS}$ by adding decorations of flags on punctures and the boundary marked points of $\bS$. These spaces have found significant applications in representation theory \cite{GS1, GS3}, higher Teichm\"uller theory \cite{FGteich}, Donaldson-Thomas theory \cite{GS2}, etc. One fundamental problem is constructing intrinsic cluster structures on the aforementioned pair of  moduli spaces, which have been achieved by Fock and Goncharov \cite{FGteich} for $\G={\rm PGL}_{m}$, and by Le \cite{Le} for classical groups based on a case-by-case study. The paper \cite{GS3} presents a universal construction of the cluster structures for all semisimple groups and solves the problem in full generality.

The paper \cite{GS3} further introduces a  moduli space $\mathscr{P}_{\G, \bS}$, which extends the space $\mathscr{X}_{\G, \bS}$ 
by adding extra data called pinnings. The pair $(\mathscr{P}_{\G, \bS}, \mathscr{A}_{\tilde{\G}, \bS})$ form a cluster ensemble under the framework of Fock and Goncharov \cite{FGensemble}. In particular, the coordinate ring of  $\mathscr{P}_{\G, \bS}$ is naturally isomorphic to a cluster Poisson algebra (\cite[Theorem 1]{Sh2}), and hence admits a natural cluster quantization $\mathcal{O}_q(\mathscr{P}_{\G, \bS})$.

In this Section, we briefly recall the definitions and several properties of $\mathscr{P}_{\G,\bS}$ and $\mathscr{A}_{\tilde{\G},\bS}$, mainly following \cite{GS3}.

\subsection{The moduli space $\mathscr{P}_{\G, \bS}$}
Let $\B$ and $\B'$ be a pair of Borel subgroups of $\G$ such that $\B\cap \B'$ is abelian. In this case, we say $\B$ and $\B'$ are of {\it generic position}. Correspondingly, we obtain a decomposition for the Lie algebra 
\be
\label{cartan.decom}
\mathfrak{g}=Lie \G = \mathfrak{u}\oplus \mathfrak{h}\oplus \mathfrak{u}^-,
\ee
such that $Lie \B= \mathfrak{u}\oplus \mathfrak{h}$ and $Lie \B'=\mathfrak{h}\oplus \mathfrak{u}^-$. A {\it pinning} $\pi$ over the pair $(\B, \B')$ is equivalent to a choice of Chevalley basis that is compatible with the decomposition \eqref{cartan.decom}.

Let $\mathcal{B}$ be the flag variety  that parametrizes the Borel subgroups of $\G$. Let $\mathcal{L}\in {\bf Hom}(\pi_1(\bS), \G)$ be a $\G$-local system over $\bS$. Consider the associated bundle $\mathcal{L}_\mathcal{B}:= \mathcal{L}\times_\G \mathcal{B}$. The marked boundary points separate the boundary $\partial \bS$ into disconnected intervals. 

\bd The moduli space $\mathscr{P}_{\G, \bS}$ parametrizes the data $(\mathcal{L}, \{\B_x\}, \{\B_p\}, \{\pi_e\})$, modulo the conjugation of $\G$, where
\begin{itemize}
\item $\mathcal{L}\in {\rm Hom}(\pi_1(\bS), \G)$ is a $\G$-local system;
\item for every marked boundary point $x$, $\B_x$ is section of $\mathcal{L}_\mathcal{B}$ over $x$;
\item for every puncture $p$, $\B_p$ is a flat section of $\mathcal{L}_\mathcal{B}$ over the circle $c_p$ surrounding $p$;
\item for every boundary interval $e$ with endpoints $a$ and $b$, the associated pair $(\B_a,\B_b)$ is generic, and $\pi_e$ is a pinning over $(\B_a,\B_b)$.
\end{itemize}
\ed

As shown in \cite[\S 13]{GS3}, every puncture $p$ corresponds to a birational Weyl group action on $\mathscr{P}_{\G, \bS}$. 
More concretely, let $m_p$ be the monodromy of a $\G$-local system $\mathcal{L}$ surrounding the puncture $p$. 
Generically, the flags invariant under the monodromy $m_p$ form a Weyl group torsor. The Weyl group acts on $\mathscr{P}_{\G, \bS}$ by alternating the flat section $\B_p$ and keeping the rest invariant.  

\subsection{The moduli space $\mathscr{A}_{\tilde{\G},\bS}$}
Consider the fiber bundle $T'\bS=T\bS\backslash \{\mbox{0-section}\}$ with fiber $\mathbb{R}^2\backslash \{0\}$, obtained from the tangent bundle $T\bS$ by deleting the 0-section. The projection from $T'\bS$ to $\bS$ induces an exact sequence
\[
0\longrightarrow \pi_1(\mathbb{R}^2\backslash\{0\}) \longrightarrow \pi_1(T'\bS) \longrightarrow \pi_1(\bS) \longrightarrow 0.
\]
The image of the generator of $\pi_1(\mathbb{R}^2\backslash\{0\})=\mathbb{Z}$ in $\pi_1(T'\bS)$ is denoted by $\mathfrak{o}$. It is known that $\mathfrak{o}$ belongs to the center of $\pi_1(T'\bS)$.

\smallskip 

There is a natural set-theoretic lift $w\mapsto \overline{w}$ from the Weyl group $W$ to $\tilde{\G}$ (cf. \cite[\S 2]{FGteich}). Set $s_{\tilde{\G}}:=\overline{w}_0^2$, where $w_0$ is the longest element in the Weyl group. The element $s_{\tilde{\G}}$ belongs to the center of $\tilde{\G}$, and $s_{\tilde{\G}}^2=1$.
A {\it twisted local system} $\mathcal{L}$ is an element of ${\bf Hom}(\pi_1(T'\bS), \tilde{\G})$ such that $\mathcal{L}(\mathfrak{o})=s_{\widetilde{\G}}$.
Every oriented loop $c_p\subset \bS$ surrounding a puncture $p$ can be uniquely lift to an oriented loop $[c_p]\subset T'\bS$. We say a twisted local system $\mathcal{L}$ is {\it unipotent} if the monodromy \[u_p:=\mathcal{L}([c_p])\] is a unipotent element in $\tilde{\G}$.

\smallskip 

Consider the decorated flag variety $\mathcal{A}:=\tilde{\G}/\U$, where $\U$ is a fixed maximal unipotent subgroup of $\tilde{\G}$. Let $\mathcal{L}_{\mathcal{A}}= \mathcal{L}\times_{\tilde{\G}}\mathcal{A}$ be its corresponding associated bundle. 
For every boundary marked point $x$, we fix a point $x'$ on the fiber $T'_x\bS$. A decoration of a twisted unipotent local system $\mathcal{L}$ is an assignment of a flat section $\A_x$ of $\mathcal{L}_\mathcal{A}$ over every $x'$, and a flat section $\A_p$ over $[c_p]$ for every puncture $p$. In particular, $\A_p$ is invariant  under the monodromy $u_p$.

\bd
The moduli space $\mathscr{A}_{\tilde{\G}, \bS}$ parametrizes the data $(\mathcal{L}, \{\A_x\}\cup \{\A_p\})$, modulo the conjugation action of $\tilde{\G}$, where $\mathcal{L}$ is a twisted unipotent local system, and  $\{\A_x\}\cup \{\A_p\}$ is a decoration of $\mathcal{L}$.
\ed

\begin{example}  Let $t$ be a triangle, i.e., a disk with three marked points on its boundary. The space $\mathscr{A}_{\tilde{\G}, t}$ is isomorphic to the configuration space ${\rm Conf}_3(\mathcal{A}):= \tilde{\G}{\Big \backslash} \mathcal{A}^3$.

Let $\tilde{\H}$ be a Cartan subgroup of $\tilde{\G}$ such that $\U\tilde{\H}$ forms a Borel subgroup. There is a natural birational map 
\[
\U\times \tilde{\H} \times \tilde{\H}\longrightarrow \mathscr{A}_{\G',t}
\]
\[
(u, h_1, h_2) \longmapsto (\U, uh_1\overline{w}_0\U, h_2\overline{w}_0\U).
\]
As illustrated by Figure \ref{eainvariants}, we refer $u$ as the {\it angle invariant} associated with the angle of the triangle $t$ near $\A_1$, and refer $h_1$ and $h_2$ as the the {\it edge invariants} associated with the oriented edges indicated by the Figure.

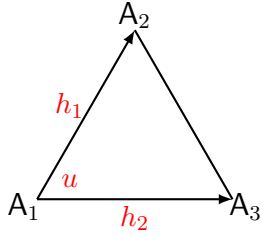
\begin{figure}[ht]
\begin{tikzpicture}
\draw[thick] (90:1.5) -- (330:1.5);
\draw[thick, ->,>=latex] (210:1.5) -- (330:1.5);
\draw[thick,->, >=latex] (210:1.5)  -- (90:1.5);
\node at (90:1.7) {$\A_2$};
\node at (210:1.7) {$\A_1$};
\node at (330: 1.7) {$\A_3$};
\node[red] at (150:1) {{\small $h_1$}};
\node[red] at (270:1) {{\small $h_2$}};
\node[red] at (210:1) {{\small $u$}};
\end{tikzpicture}
\caption{The edge and angle invariants for $\mathscr{A}_{\tilde{\G},t}$.}
\label{eainvariants}
\end{figure}

For general surface $\bS$, let $\mathcal{T}$ be an ideal triangulation of $\bS$, i.e., a triangulation whose vertices are the marked boundary points and punctures. By restriction to every triangle $t\in \mathcal{T}$, we get a projection
\[
{\rm Proj}_t: ~ \mathscr{A}_{\tilde{\G}, \bS} \longrightarrow \mathscr{A}_{\widetilde{\G},t}.
\]
Correspondingly, we get the angle and edge invariants of $\mathscr{A}_{\tilde{\G}, \bS}$ associated with $t$.
\end{example}

Let $p$ be a puncture of $\bS$. Let $\U_{\A_p}$ be the unipotent subgroup that stabilizes $\A_p$. Conversely, the choice of the decoration $\A_p$ determines a map 
\[
\U_{\A_p} \longrightarrow \U_{\A_p}/[\U_{\A_p}, \U_{\A_p}] \stackrel{\sim}{\lra}\mathbb{A}^r.
\]
\[
u \longmapsto (\chi_1(u), \ldots, \chi_r(u)),
\]
where $r$ is the rank of $\tilde{\G}$.
Note that the monodromy $u_p$ is in $\U_{\A_p}$. We set $\mathcal{W}_{p,i}=\chi_i(u_p)$. Following \cite{GS1}, we obtain  a 
{\it potential} function 
\be
\label{potenial.A}
\mathcal{W}_p=\sum_{i=1}^r \mathcal{W}_{p,i}:~ \mathscr{A}_{\tilde{\G}, \bS} \longrightarrow \mathbb{A},
\ee
constructed as follows. 

\smallskip

Following \cite[\S 6]{GS2}, for every puncture $p$, there is a birational Weyl group action on $\mathscr{A}_{\tilde{\G}, \bS}$. In details, let $\alpha_i^\vee$ be a simple positive coroot. The simple reflection $s_i$ acts on $\mathscr{A}_{\tilde{\G}, \bS}$ by rescaling the flat section $\A_p$ by $\alpha_i^\vee (\mathcal{W}_{p,i})$ and keeping the rest data invariant.

\subsection{Cluster structures} One of the main results of \cite{GS3} shows that the pair $(\mathscr{P}_{\G, \bS}, \mathscr{A}_{\tilde{\G}, \bS})$ form a cluster ensemble. Below we summarize several features of the cluster structures for future use. We refer the readers to {\it loc.cit.} for concrete constructions. 

\begin{example}
Let $t$ be a triangle. Let us pick a vertex $v$ of $t$ and a reduced word ${\bf s}$ of the longest Weyl group element $w_0$. There is a natural cluster $K_2$ (resp. Poisson) seed associated with $\mathscr{A}_{\widetilde{\G}, t}$ (resp. $\mathscr{P}_{\G, t}$). 

\smallskip
The cluster structure for $\mathscr{A}_{\widetilde{\G}, t}$ (resp. $\mathscr{P}_{\G, t}$) is closely related to the cluster structure for the double Bruhat cell $\tilde{\G}^{e,w_0}$ (resp. $\G^{e,w_0}$), studied by \cite{BFZ} (resp. \cite{FGamalgamation}). For instance, suppose $\G$ is of type $D_4$. Take the reduced word ${\bf s}=(1234)^3$, where $4$ corresponds to the vertex in the Dynkin diagram adjacent to all the rest three vertices. The black part of the quiver in Figure \ref{quiverd4} corresponds to a quiver for the double Bruhat cell $\G^{e, w_0}$. 
The word ${\bf s}$ corresponds to an order of the positive roots. The four extra red frozen vertices are attached to the black part, according to the position of the simple positive roots in the ordered sequence of positive roots. See \cite[\S 10.2]{GS3} for more details.

\begin{figure}[ht]
\begin{tikzpicture}
\begin{scope}[xshift=-3cm]
\draw[thick] (-2,1.5)--(0,2.5)--(-2,3.5);
\draw[red, thick] (-2,1.5)--(-2,3.5);
\node at (-.5,2.5) {${\bf s}$};
\end{scope}
\begin{scope}[>=latex,scale=1]
\rqstdarrow{-0.5,0.6}{-0.5,4.4};
\rqstdarrow{-0.55,0.5}{-1,3.45};
\rqstdarrow{-0.45,0.5}{0,5.4};
\rqarrow{4,4}{0,5.5};
\rqarrow{0,5.5}{4,3.1};
\rqarrow{0.5,4.75}{-0.5,4.5};
\rqarrow{-0.5,4.5}{0.5,3.75};
\rqarrow{0,4}{-1,3.5};
\rqarrow{-1,3.5}{0,3};
\rqarrow{2,1.9}{-0.5,0.5};
\rqarrow{-0.5,0.5}{2,0.95};
\foreach \i in{1,4}
{
\node at (0.5, \i+0.75) {$\square$};
\foreach \j in {0,2,4}
{\node at (\j,\i) {$\square$};}
}
\foreach \i in{2,3}
{
\node at (0.5, \i+0.75) {$\circ$};
\foreach \j in {0,2,4}
{\node at (\j,\i) {$\circ$};}
}
\foreach \i in{2,3,4}
{
\qarrow{0.5,\i-0.15}{0.5,\i+0.65};
\foreach \j in {0,2,4}
\qarrow{\j,\i-0.9}{\j,\i-0.1};
}
\foreach \i in{1,4}
{
\qdarrow{0.6,\i+0.7}{1.9,\i+0.05};
\qdarrow{0.1,\i}{1.9,\i};
\qdarrow{3.9,\i}{2.1,\i};
}
\foreach \i in{2,3}
{
\qarrow{0.6,\i+0.7}{1.9,\i+0.05};
\qarrow{0.1,\i}{1.9,\i};
\qarrow{3.9,\i}{2.1,\i};
}
\foreach \i in{2,3,4}
{
\qarrow{1.9,\i-0.05}{0.6, \i-0.25};
\qarrow{2.1,\i-0.1}{3.9,\i-0.9};
\qarrow{1.9,\i-0.1}{0.1,\i-0.9};
}
\node[red] at (-1,3.5) {$\square$};
\node[red] at (-0.5,4.5) {$\square$};
\node[red] at (-0.5,0.5) {$\square$};
\node[red] at (0,5.5) {$\square$};
\end{scope}
\end{tikzpicture}
\caption{A quiver associated with $(\mathscr{A}_{\widetilde{\G}, t}, \mathscr{P}_{{\G}, t})$, where $\G$ is of type $D_4$}
\label{quiverd4}
\end{figure}
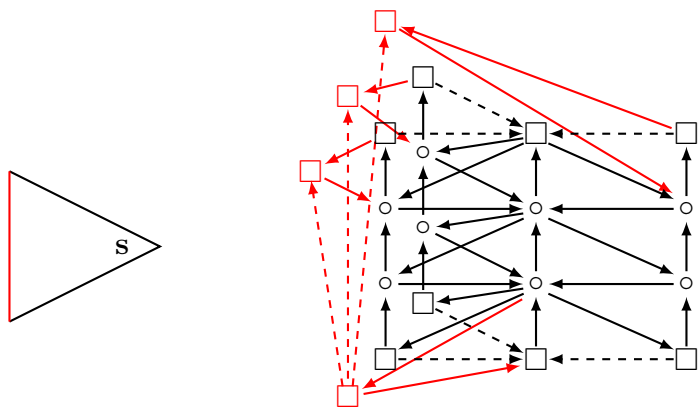

For general $\G$, the corresponding quiver $Q_{v, {\bf s}}$ has $3r$ many frozen vertices, where $r$ is the rank of $\G$. The quiver is not necessarily planar. However, we can still place them on the top of the triangle $t$ such that each side of $t$ contains $r$ frozen vertices. In particular, as in Figure \ref{quiverd4}, the extra red frozen vertices are placed on the red side of the triangle.
\end{example}
\smallskip

For general surface $\bS$, let $\mathcal{T}$ be an ideal triangulation of $\bS$. For each triangle $t$ in $\mathcal{T}$, we pick a vertex $v_t$ and a reduced word ${\bf s}_t$. The data $\widetilde{\mathcal{T}}=(\mathcal{T}, \{(v_t, {\bf s}_t)\}$ is called a {\it decorated ideal triangulation}. 
We obtain a quiver $Q_{\widetilde{\mathcal{T}}}$ by amalgamating all the local quiver ${Q}_{v_t, {\bf s}_t}$ along the corresponding edges of the $t's$.

\begin{example} Let $\G$ be of type $A_3$. We fix a reduced word ${\bf s}=123121$. The corresponding quiver for a triangle is illustrated in Figure \ref{quivera3}. Note that this is exactly the quiver constructed by Fock and Goncharov for type $A$ cases. 
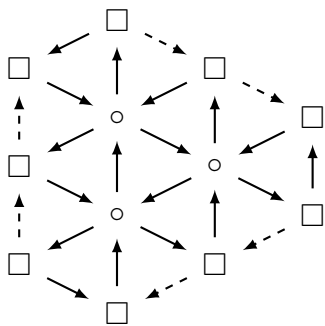
\begin{figure}[ht]
\begin{tikzpicture}[>=latex, scale=1.3]
\node at (0,0.5) (A1) {$\square$};
\node at (-1,1) (A2) {$\square$};
\node at (-2,1.5) (A3) {$\square$};
\node at (0,-0.5) (B1) {$\square$};
\node at (-1,-1) (B2) {$\square$};
\node at (-2,-1.5) (B3) {$\square$};
\node at (-1,0) (a-1) {$\circ$};
\node at (-2,0.5) (b-2) {$\circ$};
\node at (-2,-0.5) (c-2) {$\circ$};
\node at (-3,1) (d-3) {$\square$};
\node at (-3,0) (e-3) {$\square$};
\node at (-3,-1) (f-3) {$\square$};
\qarrow{A1}{a-1};
\qarrow{a-1}{B1};
\qarrow{A2}{b-2};
\qarrow{c-2}{b-2};
\qarrow{c-2}{B2};
\qarrow{A3}{d-3};
\qarrow{d-3}{b-2};
\qarrow{b-2}{e-3};
\qarrow{e-3}{c-2};
\qarrow{c-2}{f-3};
\qarrow{f-3}{B3};
\qarrow{b-2}{A3};
\qarrow{B3}{c-2};
\qarrow{a-1}{A2};
\qarrow{B2}{a-1};
\qarrow{b-2}{a-1};
\qarrow{a-1}{c-2};
\qarrow{B1}{A1};
\qdarrow{f-3}{e-3};
\qdarrow{e-3}{d-3};
\qdarrow{A3}{A2};
\qdarrow{A2}{A1};
\qdarrow{B1}{B2};
\qdarrow{B2}{B3};
\end{tikzpicture}
\caption{Quiver for $A_3$.}
\label{quivera3}
\end{figure}

Now let $\odot$ be a punctured disk with two marked points on its boundary. We pick a decorated triangulation as shown on the left graph of Figure \ref{odotquiver}. Let us place a copy of the quiver in Figure \ref{quivera3} on the top of the left triangle in $\odot$. We make another copy of the quiver, rotate it by $180^\circ$, and place it on the right triangle. We sure identify the frozen vertices on the same internal edges and make them mutable. Note that the arrows among those merged vertices are of opposite orientations and therefore get canceled. In the end, we obtain the quiver on the right hand side for $(\mathscr{A}_{SL_4, \odot}, \mathscr{P}_{PGL_4, \odot})$

\begin{figure}[ht]
\begin{tikzpicture}
\begin{scope}[scale=1.2]
\node at (-7,0) {$\bullet$};
\node at (-7, 1.5) {$\bullet$};
\node at (-7, -1.5) {$\bullet$};
\draw (-7, 0) circle (1.5cm);
\draw (-7, -1.5) -- (-7, 1.5);
\node at (-7.3,0) {${\bf s}$};
\node at (-6.7, 0) {${\bf s}$};
\end{scope}
\begin{scope}[>=latex,scale=1.2]
\node at (0,0.5) (A1) {$\circ$};
\node at (0,1.5) (A2) {$\circ$};
\node at (0,2.5) (A3) {$\circ$};
\node at (0,-0.5) (B1) {$\circ$};
\node at (0,-1.5) (B2) {$\circ$};
\node at (0,-2.5) (B3) {$\circ$};
\foreach \i in{-1,1}
{
\node at (\i,0) (a\i) {$\circ$};
}
\foreach \i in{-2,2}
{
\node at (\i,0.5) (b\i) {$\circ$};
\node at (\i,-0.5) (c\i) {$\circ$};
}
\foreach \i in{-3,3}
{
\node at (\i,1) (d\i) {$\square$};
\node at (\i,0) (e\i) {$\square$};
\node at (\i,-1) (f\i) {$\square$};
}
\node at (0,0.5) (A1) {$\circ$};
\node at (0,1.5) (A2) {$\circ$};
\node at (0,2.5) (A3) {$\circ$};
\node at (0,-0.5) (B1) {$\circ$};
\node at (0,-1.5) (B2) {$\circ$};
\node at (0,-2.5) (B3) {$\circ$};
\qarrow{A1}{a-1};
\qarrow{a-1}{B1};
\qarrow{B1}{a1};
\qarrow{a1}{A1};
\qarrow{A2}{b-2};
\qarrow{c-2}{b-2};
\qarrow{c-2}{B2};
\qarrow{B2}{c2};
\qarrow{b2}{c2};
\qarrow{b2}{A2};
\qarrow{A3}{d-3};
\qarrow{d-3}{b-2};
\qarrow{b-2}{e-3};
\qarrow{e-3}{c-2};
\qarrow{c-2}{f-3};
\qarrow{f-3}{B3};
\qarrow{B3}{f3};
\qarrow{f3}{c2};
\qarrow{c2}{e3};
\qarrow{e3}{b2};
\qarrow{b2}{d3};
\qarrow{d3}{A3};
\qarrow{b-2}{A3};
\qarrow{A3}{b2};
\qarrow{c2}{B3};
\qarrow{B3}{c-2};
\qarrow{a-1}{A2};
\qarrow{A2}{a1};
\qarrow{a1}{B2};
\qarrow{B2}{a-1};
\qarrow{b-2}{a-1};
\qarrow{a-1}{c-2};
\qarrow{c2}{a1};
\qarrow{a1}{b2};
\qdarrow{f-3}{e-3};
\qdarrow{e-3}{d-3};
\qdarrow{d3}{e3};
\qdarrow{e3}{f3};
\end{scope}
\end{tikzpicture}
\caption{A quiver for  $(\mathscr{A}_{SL_4, \odot},\mathscr{P}_{PGL_4, \odot})$}
\label{odotquiver}
\end{figure}
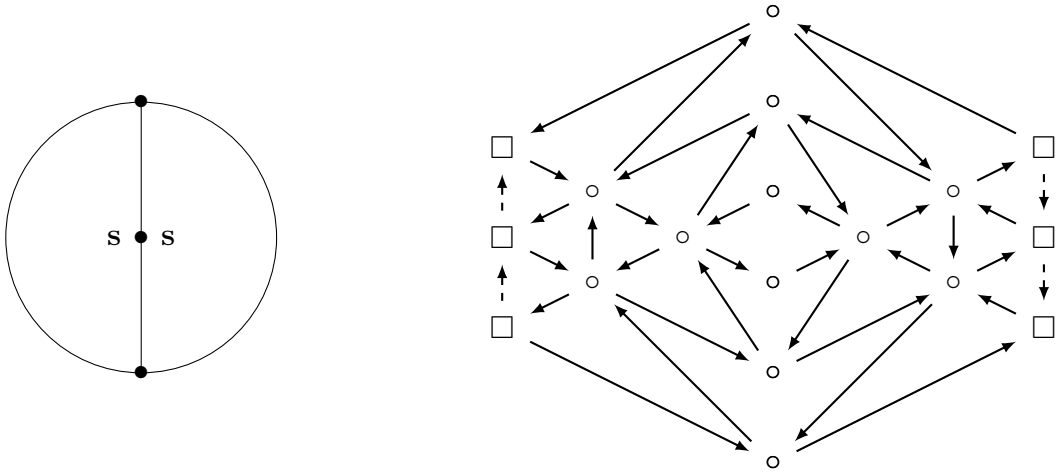
\end{example}

One may further assign coordinates to vertices of the quiver $Q_{\widetilde{\mathcal{T}}}$. Hence, each decorated triangulation $\mathcal{T}$ gives rise to a cluster $K_2$ seed for $\mathscr{A}_{\tilde{\G},\bS}$ and a cluster Poisson seed for $\mathscr{P}_{\G, \bS}$. By Theorem 5.11 of \cite{GS3}, the cluster seeds obtained from different decorated triangulations are mutation equivalent. Therefore, we obtain a natural cluster ensemble structure on $(\mathscr{P}_{\G, \bS}, \mathscr{A}_{\tilde{\G}, \bS})$.

\smallskip 
Let $\mathcal{O}_q(\mathscr{P}_{\G, \bS})$ be the Fock-Goncharov quantum cluster algebra associated with the underlying cluster structure of $\mathscr{P}_{\G, \bS}$.
Let $\G^L$ be the Langlands dual of $\G$. Barring a few exceptions, the paper \cite{GS2, GS3} show that the Donaldson Thomas transformation of $\mathscr{P}_{\G,\bS}$ is a cluster transformation. Combining with  Theorem \ref{quantum.cluster.dual}, we get the following result.

\bt Let us exclude surfaces $\bS$ with exactly one puncture and no boundaries. Let $\G$ be a simply-laced semisimple algebraic group over $\Q$ with trivial center.
The Fock-Goncharov quantum cluster algebra $\mathcal{O}_q(\mathscr{P}_{\G, \bS})$ has a quantum theta basis equivariently parametrized by the tropical points in $\mathscr{A}_{\G^L, \bS}(\mathbb{Z})$. The parametrization is equivarient under the following group actions:
\begin{itemize}
    \item the mapping class group of $\bS$,
    \item the outer group of $\G$,
    \item the product of Weyl groups over punctures of $\bS$,
    \item the product of braid groups over boundary circles of $\bS$.
\end{itemize}
\et 

\section{Cluster Realization of Quantum Groups}
In this Section, we focus on the cases when $\bS=\odot$ is a once punctured disk with two marked points. The group $\G$ is adjoint and simply-laced. We provide a rigid cluster model realizing the quantum group ${\bf U}_q(\g)$.

\subsection{The moduli space $\mathscr{P}_{\G, \odot}$}\la{PDOT}
Let us fix a pair $\B^+, \B^-$ of opposite Borel subgroups of $\G$. Let ${\rm H}=\B^+\cap \B^-$. We fix a pinning $\pi_{\rm std}$ over $(\B^+, \B^-)$ and refer it as the standard pinning. Equivalently, we obtain a Chevalley basis $\{e_{\alpha}, f_{\alpha}, h_{\alpha_i}\}$ for the Lie algebra $\mathfrak{g}=Lie \G$.

\smallskip 

Let ${\U}$ be the maximal unipotent subgroup inside $\B^+$. We have the decomposition $\B^+={\U}{\H}$. 
 For $1\leq i\leq r$, there are additive characters
\be
\label{additive.character}
\chi_i: ~{\U}\longrightarrow \mathbb{A}^1
\ee
such that $\chi_i({\rm exp}\, t e_{\alpha_j})=\delta_{ij}t$, where $\delta_{ij}$ is the Kronecker symbol.

\begin{figure}[ht]
\begin{tikzpicture}
\begin{scope}[scale=1]
\node at (0,0) {$\bullet$};
\node at (0, 1.5) {$\bullet$};
\node at (0, -1.5) {$\bullet$};
\draw (0, 0) circle (1.5cm);
\node at (-1.7,0) {$\pi$};
\node at (1.7, 0) {$\pi'$};
\draw[->,>=stealth',semithick, yshift=-3.5cm, red] (110:4cm) arc (110:70:4cm);
\draw[<-,>=stealth',semithick, yshift=3.5cm, red] (-110:4cm) arc (-110:-70:4cm);
\node[red] at (0,0.8) {{\tiny $b_1\overline{w}_0$}};
\node[red] at (0,-0.8) {{\tiny $b_2\overline{w}_0$}};
\node at (0.4,0) {{\tiny $\B_p$}};
\end{scope}
\end{tikzpicture}
\caption{The moduli space $\mathscr{P}_{\G, \odot}$}
\label{modulip}
\end{figure}
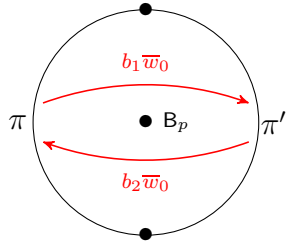

Recall the moduli space $\mathscr{P}_{\G, \odot}$.
As illustrated by Figure \ref{modulip}, there are pinnings $\pi$ and $\pi'$ associated with the two boundary intervals of $\odot$. The parallel transport $b_1\overline{w}_0$ takes $\pi$ to $\pi'$ along the top path, and the parallel transport $b_2\overline{w}_0$ takes $\pi'$ to $\pi$ along the bottom path, where $b_1, b_2\in \B^+$. The flat section associated with the puncture corresponds to a Borel subgroup $\B_p$ containing the element $b_1\overline{w}_0b_2\overline{w}_0$. In this way, we obtain an isomorphism
\[
\mathscr{P}_{\G, \odot}=\left\{(b_1, b_2, \B_p)\in \B^+\times \B^+\times \mathcal{B} ~\middle | ~b_1\overline{w}_0b_2\overline{w}_0\in \B_p \right\}.
\]

Let $b_{\epsilon}=u_{\epsilon}h_{\epsilon}\in {\U}{\H}$ for $\epsilon=1,2$. We get a set of regular functions $\{\mathcal{W}_{\epsilon, i}, \mathcal{K}_{\epsilon, i}\}_{1\leq i\leq r}$ defined via 
\[
\mathcal{W}_{\epsilon, i} =\chi_i(u_\epsilon),\hskip 12mm \mathcal{K}_{\epsilon,i}=\alpha_i(h_\epsilon).
\]
The functions $\mathcal{W}_{\epsilon, i} $ and $\mathcal{K}_{\epsilon,i}$ are global monomials (\cite[\S 15.3]{GS3}).

Recall the Weyl group action on $\mathscr{P}_{\G, \odot}$ that only alternates the Borel subgroup $\B_p$ and keeps the rest invariant. Therefore the functions $\mathcal{W}_{\epsilon, i} $ and $\mathcal{K}_{\epsilon,i}$ are invariant under the Weyl group action. By \cite{GS2,GS3}, the Weyl group action are cluster automorphisms of $\mathscr{P}_{\G, \odot}$. In other words, the Weyl group $W$ can be embedded into the (quasi-)cluster modular group $\mathcal{G}$ for $\mathscr{P}_{\G, \odot}$. As a consequence, the Weyl group $W$ acts on the algebra $\mathcal{O}_q(\mathscr{P}_{\G, \odot})$ and permutes its quantum theta basis.  Denote by 
$\mathcal{O}_q(\mathscr{P}_{\G, \odot})^W$ the $W$-invariant subalgebra of $\mathcal{O}_q(\mathscr{P}_{\G, \odot})$.
Following the quantum lift Theorem \cite[\S 18]{GS3}, the functions $\mathcal{W}_{\epsilon, i}$ and $\mathcal{K}_{\epsilon,i}$ can be uniquely promoted to quantum functions 
\be
\label{globalWK}
\mathbb{W}_{\epsilon, i}, ~\mathbb{K}_{\epsilon,i}\in \mathcal{O}_q(\mathscr{P}_{\G, \odot})^W, \qquad \epsilon=1,2, ~~~ 1\leq i\leq r.
\ee

\smallskip

Below we present a pure cluster interpretation of the quantized functions in \eqref{globalWK}. Recall the quantum duality map
\be
\la{bijections.a}
\mathscr{A}_{\G^L, \odot}(\mathbb{Z}^t)\stackrel{\sim}{\longrightarrow} \Theta(\mathcal{O}_q(\mathscr{P}_{\G,\odot})).
\ee
Recall the angle and edge invariants associated with $\mathscr{A}_{\G^L, \odot}$, as illustrated by Figure \ref{moduliA}. 
\begin{figure}[ht]
\begin{tikzpicture}
\begin{scope}[scale=1]
\node at (0, 1.5) {$\bullet$};
\node at (0, 1.75) {\footnotesize $\A_e$};
\node at (0, -1.5) {$\bullet$};
\node at (0, -1.75) {\footnotesize $\A_f$};
\draw[->,>=stealth',semithick, red] (95:0.5cm) arc (95:265:0.5cm);
\draw[<-,>=stealth',semithick, red] (91.5:1.5cm) arc (91.5:268:1.5cm);
\draw[<-,>=stealth',semithick, red] (85:0.5cm) arc (85:-85:0.5cm);
\draw[<-,>=stealth',semithick, red] (-88:1.5cm) arc (-88:88:1.5cm);
\draw[->,>=stealth',semithick, red] (0,0) -- (0,1.45);
\draw[->,>=stealth',semithick, red] (0,0) -- (0,-1.45);
\node at (-0.7,0) {{\tiny $u_1$}};
\node at (0.7,0) {{\tiny $u_2$}};
\node at (-.2,0.8) {{\tiny $h_1$}};
\node at (.2, -0.8) {{\tiny $h_2$}};
\node at (-1.7,0) {{\tiny $h_3$}};
\node at (1.7,0) {{\tiny $h_4$}};
\node at (0,0) {$\bullet$};
\node at (0.25, 0) {\footnotesize $\A_p$};
\end{scope}
\end{tikzpicture}
\caption{The edge and angle invariants for the moduli space $\mathscr{A}_{\G^L, \odot}$}
\label{moduliA}
\end{figure}
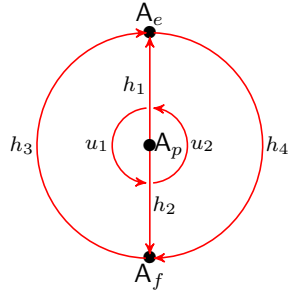

Let us fix a decorated triangulation $\widetilde{\mathcal{T}}$ for $\odot$, as the left graph of Figure \ref{odotquiver}. The corresponding quiver $Q_{\mathcal{T}}$ is placed on the top of $\odot$, where each edge of $\mathcal{T}$ contains $r$ vertices. The cluster $K_2$ variables of $\mathscr{A}_{\G^L, \odot}$ associated with these $4r$ vertices are shown on Figure \ref{Quiver.moduliA2}. 
In terms of the edge invariants, we have
\[
A_{e,i}=\omega_i(h_1^{-1}), \qquad A_{f,i}=\omega_i(h_2^{-1}), \qquad 
A_i=\omega_i(h_3^{-1}), \qquad A_{i}'=\omega_{i}(h_4^{-1}),
\]
where $\omega_i$  are the fundamental weights of $\G^L$. 

\begin{figure}[ht]
\begin{tikzpicture}
\begin{scope}[scale=2]
\node at (0, 1.5) {$\bullet$};
\node at (0, -1.5) {$\bullet$};
\draw (0, 0) circle (1.5cm);
\draw (0,0) -- (0,1.45);
\draw (0,0) -- (0,-1.45);
\node[red,thick] at (0,1.2) {{\small $\circ$}};
\node[red,thick] at (0,0.9) {{\small $\circ$}};
\node[red,thick] at (0.05,0.7) {{\small $\vdots$}};
\node[red,thick] at (0,0.3) {{\small $\circ$}};
\node[red,thick] at (0,-1.2) {{\small $\circ$}};
\node[red,thick] at (0,-0.9) {{\small $\circ$}};
\node[red,thick] at (0.05,-0.5) {{\small $\vdots$}};
\node[red,thick] at (0,-0.3) {{\small $\circ$}};
\node[red,thick] at (120:1.5) {{\tiny $\square$}};
\node[red,thick] at (150:1.5) {{\tiny $\square$}};
\node[red,thick] at (180:1.6) {{\tiny $\vdots$}};
\node[red,thick] at (-120:1.5) {{\tiny $\square$}};
\node[red,thick] at (-150:1.5) {{\tiny $\square$}};
\node[red,thick] at (120:-1.5) {{\tiny $\square$}};
\node[red,thick] at (150:-1.5) {{\tiny $\square$}};
\node[red,thick] at (180:-1.6) {{\tiny $\vdots$}};
\node[red,thick] at (-120:-1.5) {{\tiny $\square$}};
\node[red,thick] at (-150:-1.5) {{\tiny $\square$}};
\node at (0,0) {$\bullet$};
\node at (0.2,1.2) {{\small $A_{e,1}$}};
\node at (0.2,0.9) {{\small $A_{e,2}$}};
\node at (0.2,0.3) {{\small $A_{e,r}$}};
\node at (0.2,-1.2) {{\small $A_{f,1}$}};
\node at (0.2,-0.9) {{\small $A_{f,2}$}};
\node at (0.2,-0.3) {{\small $A_{f,r}$}};
\node at (120:1.7) {{\small $A_1$}};
\node at (150:1.7) {{\small $A_2$}};
\node at (-120:1.7) {{\small $A_r$}};
\node at (-150:1.7) {{\small $A_{r-1}$}};
\node at (120:-1.7) {{\small $A_1'$}};
\node at (150:-1.7) {{\small $A_2'$}};
\node at (-120:-1.7) {{\small $A_r'$}};
\node at (-150:-1.7) {{\small $A_{r-1}'$}};
\end{scope}
\end{tikzpicture}
\caption{A collection of cluster variables for the moduli space $\mathscr{A}_{\G^L, \odot}$}
\label{Quiver.moduliA}
\end{figure}
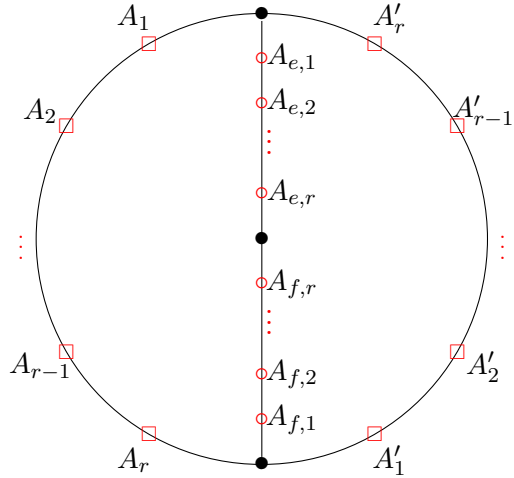

The variables $A_i, A_{i}'$ are placed on the boundary intervals and are frozen. By \eqref{frozen-theta-function}, there is a unique tropical point $l_i$ such that $l_i(A_i)=1$ and $l_i(A)=0$ for all the rest cluster variables. Similarly, there is a tropical point $l_i'$ for $A_{i}'$. 
By the quantum duality map \eqref{bijections.a}, we obtain the quantum theta functions 
\[\theta_{l_i}, ~\theta_{l_i'} \in \Theta(\mathcal{O}_q(\mathscr{P}_{\G,\odot})).\]

\smallskip

For $1\leq i \leq r$, define the positive functions
\be
\mathfrak{w}_i: ~{\H}^L \longrightarrow \mathbb{G}_m, \qquad h \longmapsto \omega_i(h^{-1}).
\ee
Their tropicalization defines bijection
\be
\label{tropisweight0}
\mathfrak{w}^t:~\H^L(\mathbb{Z}^t) \stackrel{\sim}{\longrightarrow}\mathbb{Z}^r,\qquad x \longmapsto \left(\mathfrak{w}_1^t(x),\ldots, \mathfrak{w}_r(x)\right).
\ee
We further identify ${\H}^L(\mathbb{Z}^t)$ with  the character lattice $X^\ast(\H)$ of ${\rm H}$:
\be
\label{tropisweight}
\H^L(\mathbb{Z}^t) \stackrel{\sim}{\longrightarrow} X^\ast(\H),\qquad x \longmapsto  \sum_{i=1}^r \mathfrak{w}_i^t(x) \alpha_i.
\ee
The invariants in Figure \ref{moduliA} define a positive birational map 
\[
\mathscr{A}_{\G^L, \odot}\longrightarrow \U^L\times \H^L\times \U^L \times \H^L
\]
\[
(\mathcal{L}, \{\A_e, \A_f\}\cup\{\A_p\}) \longmapsto (u_1, h_1, u_2, h_2)
\]
Its tropicalization combined with \eqref{tropisweight} yields a bijection
\be
\label{acnoanbit}
\tau:~\mathscr{A}_{\G^L, \odot}(\mathbb{Z}^t) \stackrel{\sim}{=} \U^L(\mathbb{Z}^t)\times X^*(\H)\times \U^L(\mathbb{Z}^t) \times X^*(\H).
\ee
Let $k_i, k_i'\in \mathscr{A}_{\G^L, \odot}(\mathbb{Z}^t)$ be the tropical points such that 
\[
\tau(k_i)=(0, \alpha_i, 0,0), \qquad \tau(k_i)=(0, 0, 0,\alpha_i).
\]
Correspondlingly, we get the quantum theta functions 
\[
\theta_{k_i}, \theta_{k_i'}\in \Theta\left(\mathcal{O}_q(\mathscr{P}_{\G,\odot})\right).
\]

\smallskip 

Recall the potential function $\mathcal{W}_{p}=\sum \mathcal{W}_{p,i}$ associated with the puncture $p$ of $\odot$. In terms of the angle invariants $u_1$ and $u_2$ in Figure \ref{moduliA} and the additive character in \eqref{additive.character}, we have
\be
\la{partial.potential}
\mathcal{W}_{p,i} =\chi_i(u_1) + \chi_i(u_2).
\ee
Therefore $\mathcal{W}_{p}$ is a positive function. We define 
\[
\mathscr{A}^+_{\G^L,\odot}(\mathbb{Z}^t):=\left\{ l ~\middle|~ \mathcal{W}_p^t(l)\geq 0\right\}.
\]
Consider the set 
\[
\U^L_+(\mathbb{Z}):= \left\{ l ~\middle|~\chi_i^t(l)\geq 0, ~\, \forall i \right\}.
\]
Tropical points in $\U^L_+(\mathbb{Z})$ are usually refered as Lusztig data in the literature. By definition, we have the following bijection
\be
\label{trop.c.a}
\tau:~\mathscr{A}_{\G^L, \odot}^+(\mathbb{Z}^t) \stackrel{\sim}{=} \U^L_+(\mathbb{Z}^t)\times \H^L(\mathbb{Z}^t)\times \U^L_+(\mathbb{Z}^t) \times \H^L(\mathbb{Z}^t)
\ee

\smallskip 

Recall the Weyl group $W$ acts on $\mathscr{A}_{\G^L, \odot}$ by rescaling the decorated flag $\A_p$. It is known that this Weyl group action gives cluster automorphism. Therefore, $W$ acts on the tropical set $\mathscr{A}_{\G^L,\odot}(\Z^t)$. Furthermore, the bijection \eqref{bijections.a} in equivarient under the Weyl group action. 

\begin{lemma}
Every $W$-orbit $c\subset \mathscr{A}_{\G^L, \odot}(\mathbb{Z}^t)$ contains a unique element $l \in \mathscr{A}^+_{\G^L,\odot}(\mathbb{Z}^t)$.
\end{lemma}
\begin{proof}
Recall that $\omega_1,\ldots, \omega_r$ are the fundamental coweights of $\G$. Following \cite[\S 6]{GS2}, there is a $W$-equivarient map from $\mathscr{A}_{\G^L, \odot}(\mathbb{Z}^t)$ to the coweight lattice $X_*(\H)$ of $\G$:
\[
\mathscr{A}_{\G^L, \odot}(\mathbb{Z}^t) \longrightarrow X_*(\H), \qquad l \longmapsto \sum_{i=1}^r \mathcal{W}_{p,i}^t(l)\omega_i.
\]
The set $\mathscr{A}_{\G^L, \odot}^+(\mathbb{Z}^t)$ is the preimage of the dominant chamber $X_*^+(\H)$, which concludes the proof of the Lemma.
\end{proof}

\bd For $l\in \mathscr{A}^+_{\G^L, \odot}(\mathbb{Z}^t)$, we define the {\it W-invariant quantum theta fucntion}
\be
\vartheta_l:=\sum_{l'\in W\cdot l}\theta_{l'} \in \mathcal{O}_q(\mathscr{P}_{\G,\odot})^W.
\ee
\ed
\begin{lemma}
\label{lemma12d}
The set $\left\{ \vartheta_l\right\}$ is a $\mathbb{L}$-linear basis of $\mathcal{O}_q(\mathscr{P}_{\G, \odot})^W.$
\end{lemma}
\begin{proof} It is clear that the set $\left\{ \vartheta_l\right\}$ is linearly independent. 
Let $\phi$ and $\varphi$ be two elements in $\mathcal{O}_q(\mathscr{P}_{\G, \odot})^W.$ Since $\{\theta_l\}$ is a $\mathbb{L}$-linear basis of $\mathcal{O}_q(\mathscr{P}_{\G, \odot})$, we get a decomposition
\[
\phi\cdot \varphi = \sum_{l\in \mathscr{A}_{\G^L, \odot}(\mathbb{Z}^t)} c(l) \theta_l.
\]
Note that $\phi\cdot \varphi$ is $W$-invariant. Therefore we have $c(w\cdot l)=c(l)$ for any $w\in W$. As a consequence, we get 
\be
\label{sum.expr}
\sum_{l\in \mathscr{A}_{\G^L, \odot}(\mathbb{Z}^t)} c(l) \theta_l= \sum_{l\in \mathscr{A}_{\G^L, \odot}^+(\mathbb{Z}^t)} c(l) \vartheta_l,
\ee
which concludes the proof of the Lemma.
\end{proof}

\bl 
\label{check2}
The structure coefficients of the basis $\{\varepsilon_l\}$ are in $\mathbb{N}[q^{\frac{1}{2}}, q^{-\frac{1}{2}}]$.
\el
\begin{proof} Let $l_1, l_2\in \mathscr{A}_{\G^L,\odot}^+(\mathbb{Z}^t)$. We have
\[
\vartheta_{l_1}\cdot \vartheta_{l_2}= \left(\sum_{l\in W\cdot l_1}\theta_l\right)\left(\sum_{l\in W\cdot l_2}\theta_l\right) = \sum_{l\in \mathscr{A}^+_{\G^L, \odot}(\mathbb{Z}^t)}\left(\sum_{{ {\tiny \begin{array}{l}l'\in W\cdot l_1; \\ l''\in W\cdot l_2\end{array}}}} c_q(l',l''; l)\right) \vartheta_{l}.
\]
Therefore the structural coefficients 
\[\widetilde{c}_q(l_1, l_2; l):= \sum_{{\tiny \begin{array}{l}l'\in W\cdot l_1; \\ l''\in W\cdot l_2\end{array}}} c_q(l',l''; l) \in \mathbb{N}[q^{\frac{1}{2}}, q^{-\frac{1}{2}}]. \qedhere
\]
\end{proof}

\bt
\label{basic.theta.q}
 The tropical points $l_i, l_i', k_i, k_i'$ are invariant under the Weyl group action. We have
\[
\mathbb{W}_{1,i}=\theta_{l_i}=\vartheta_{l_i}, \qquad \mathbb{W}_{2,i}=\theta_{l_i'}=\vartheta_{l_i'};
\]
\[
\mathbb{K}_{1,i}=\theta_{k_i}=\vartheta_{k_i}, \qquad
\mathbb{K}_{2,i}=\theta_{k_i'}=\vartheta_{k_i'}.
\]
\et
\begin{proof} The Weyl group $W$ acts on $\mathscr{A}_{\G^L, \odot}$ via cluster transformations. Hence they map cluster variables to cluster variables and keep the frozen ones invariant. Therefore, the tropical points $l_i$ and $l_i'$ are $W$ invariant.
Meanwhile, since the $\U^L(\mathbb{Z}^t)$ parts of $k_i$ and $k_i'$ are trivial, we have
\[
\mathcal{W}_{p, j}^t(k_i)=\mathcal{W}_{p, j}^t(k_i')=0.
\]
By the definition of the Weyl group action, the tropical points $k_i$ and $k_i'$ are $W$-invariant. 

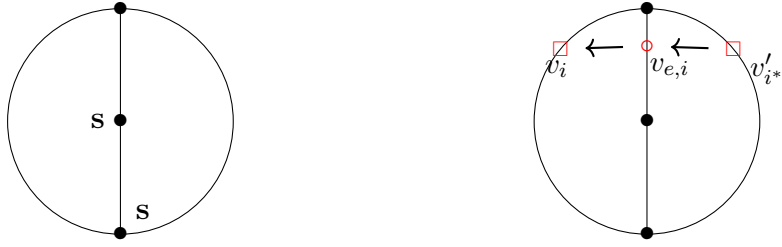
\begin{figure}[ht]
\begin{tikzpicture}
\begin{scope}
\node at (-7,0) {$\bullet$};
\node at (-7, 1.5) {$\bullet$};
\node at (-7, -1.5) {$\bullet$};
\draw (-7, 0) circle (1.5cm);
\draw (-7, -1.5) -- (-7, 1.5);
\node at (-7.3,0) {${\bf s}$};
\node at (-6.7, -1.2) {${\bf s}$};
\end{scope}
\begin{scope}
\node at (0, 1.5) {$\bullet$};
\node at (0, -1.5) {$\bullet$};
\draw (0, 0) circle (1.5cm);
\draw (0,0) -- (0,1.45);
\draw (0,0) -- (0,-1.45);
\node[red,thick] at (0,1) (C) {{\small $\circ$}};
\
\node[red,thick] at (140:1.5) (A) {{\tiny $\square$}};
\node[red,thick] at (40:1.5) (B) {{\tiny $\square$}};
\node at (0,0)  {$\bullet$};
\node at (0.3,0.7) {{\small $v_{e,i}$}};

\node at (-1.2,0.7) {{\small $v_i$}};
\node at (1.6,0.7) {{\small $v_{i^\ast}'$}};
\qarrow{B}{C};
\qarrow{C}{A};
\end{scope}
\end{tikzpicture}
\caption{The decorated ideal triangulation $\widetilde{\mathcal{T}}$.}
\label{Quiver.moduliA2}
\end{figure}

Let ${\bf s}$ be a reduced word of $w_0$ starting with $i$. We consider the decorated triangulation $\widetilde{\mathcal{T}}$ as shown on the left graph of Figure \ref{Quiver.moduliA2}.  As shown on the right graph, the quiver $Q_{\widetilde{\mathcal{T}}}$ contains three vertices $v_i, v_{e,i}, v_{i^*}'$, corresponding to the cluster variables $A_i, A_{e,i}, A_{i^*}'$ respectively. For simplicity, let us write them as 
\[
e_1:=v_i, \qquad e_2:=v_{e,i}, \qquad e_3:=v_{i^\ast}'. 
\]
By Lemma 8.9 and Lemma 15.9 of \cite{GS3}, we get 
\be
\mathbb{W}_{1,i}= X_{e_1}+ X_{e_1+e_2}, \qquad \mathbb{K}_{1,i}= X_{e_1+e_2+e_3}.
\ee
By looking at the leading terms, we prove the second half of the Theorem. The same argument applies to $\mathbb{W}_{2,i}$ and $\mathbb{K}_{2,i}$.
\end{proof}

\subsection{A grading of $\mathcal{O}_q(\mathscr{P}_{\G, \odot})^W$.} 
The Cartan subgroup ${\H}\times{\H}$ acts on $\mathscr{P}_{\G, \odot}$ by rescaling the pinnings, which makes $\mathcal{O}_q(\mathscr{P}_{\G, \odot})^W$ a $\left(X^\ast({\rm H})\times X^\ast({\rm H})\right)$-graded algebra. Below we present the grading on $\mathcal{O}_q(\mathscr{P}_{\G, \odot})^W$ in terms of the quantum cluster duality.

\smallskip

Recall the cluster frozen variables $A_i, A_i'$  of $\mathscr{A}_{\G^L, \odot}$. 
Define the weight map 
\[
{\rm wt}=({\rm wt}_1, {\rm wt}_2):~ \mathscr{A}_{\G^L, \odot}(\mathbb{Z}^t) \longrightarrow X^\ast({\rm H})\times X^\ast({\rm H}),
\]
\[
l \longmapsto \left( \sum_{i=1}^r l(A_i)\alpha_i, \, \sum_{i=1}^rl(A_i')\alpha_i \right).
\]
For $\lambda \in X^*(\H)\times X^*(\H)$, let us set 
\[
\mathcal{O}_q\left(\mathscr{P}_{\G, \odot}\right)^W_{\lambda} = \bigoplus_l \mathbb{L}\vartheta_l,\] 
where $l\in \mathscr{A}_{\G^L, \odot}^+(\mathbb{Z}^t)$ and ${\rm wt}(l)=\lambda.$ 
\bl \label{tech.lemma} The decomposition \[
\mathcal{O}_q\left(\mathscr{P}_{\G, \odot}\right)^W = \bigoplus_{\lambda\in X^*(\H)\times X^*(\H)}  \mathcal{O}_q\left(\mathscr{P}_{\G, \odot}\right)^W_\lambda, 
\] 
makes $\mathcal{O}_q\left(\mathscr{P}_{\G, \odot}\right)^W$ a $X^\ast({\rm H})\times X^\ast({\rm H})$-graded algebra. 
\el
\begin{proof} Consider the $\mathbb{L}$-linear span
\[
\mathcal{O}_q\left(\mathscr{P}_{\G, \odot}\right)_{\lambda} = \bigoplus_l \mathbb{L}\theta_l,
\]
where $l\in \mathscr{A}_{\G^L, \odot}(\mathbb{Z}^t)$ and ${\rm wt}(l)=\lambda$. 

The fact that the frozen variables are invariant under the Weyl group action implies that  
\[{\rm wt}(w\cdot l)={\rm wt}(l), \qquad  \forall w\in W, \, \forall l \in \mathscr{A}_{\G^L, \odot}(\mathbb{Z}^t).
\]
Therefore, the set $\mathcal{O}_q\left(\mathscr{P}_{\G, \odot}\right)^W_{\lambda}$ coincides with the $W$-invariants of $\mathcal{O}_q\left(\mathscr{P}_{\G, \odot}\right)_{\lambda}$. The rest is a direct consequence of Lemma \ref{decomp.quant.cl}.
\end{proof}

\subsection{A filtration of $\OPW_\lambda$.}
Note that each $\OPW_\lambda$ is still infinite-dimensional.  Below we introduce  a further filtration on $\OPW_\lambda$. 

\smallskip

Recall the cluster variables $A_{e,i}$ and $A_{f,i}$ on Figure \ref{Quiver.moduliA}. 
Let us set 
\[
F= A_{e,1}+\ldots+ A_{e, r}+ A_{f,1}+\ldots+A_{f,r}.
\]
Its tropicalization becomes
\[
F^t=\min \{ A_{e,1}^t, \ldots, A_{f,r}^t\}.
\]

\bd
For every $n\in \mathbb{Z}$, we define the following subset of $\OPW_\lambda$
\[
\mathcal{F}^n\OPW_\lambda =\bigoplus_{l} \mathbb{L}\vartheta_l, 
\]
where the direct sum is over the set
\be
\label{para.set2}
\left\{ l \in \mathscr{A}_{\G^L,\odot}^+(\mathbb{Z}^t) ~\middle|~
    {\rm wt}(l)=\lambda, ~ F^t(l)\geq -n
\right\}.
\ee
To avoid cumbersome notations, we  write
\[
\mathcal{F}_\lambda^n:= \mathcal{F}^n\OPW_\lambda.
\]
We get a filtration
\[
\cdots\mathcal{F}_\lambda^{-1}\subset \mathcal{F}_\lambda^0 \subset \mathcal{F}_\lambda^1 \subset \cdots \subset \OPW_\lambda.
\]
\ed

The following Lemma serves as  one of the main technical results for this paper. 
\begin{lemma}
\label{tech.lemma1}
The multiplication of $\OPW$ is compatible with the filtration
\[
m: \mathcal{F}_\lambda^n \times \mathcal{F}_\mu^m \longrightarrow \mathcal{F}_{\lambda+\mu}^{m+n}.
\]
\end{lemma}
\begin{proof}

The Weyl group acts on $\mathscr{A}_{\G^L, \odot}$ by rescaling the decorated flag $\A_p$ associated with the puncture $p$ and keeping the rest invariant. Let $w\in W$. For every fundamental coweight $\omega_i$ of $\G$, we consider
\[
\omega_i -w\cdot \omega_i =\sum_{k=1}^r c_k \alpha_k^\vee,
\]
where $c_k$ are non-negative integers. Correspondingly, under the action of $w$, the cluster variable $A_{e, i}$ 
becomes\footnote{For $\G=PGL_n$, the change of arbitrary cluster $K_2$ variables under the action of $W$ has been discussed in the recent work of Fraser and Pylyavskyy \cite{FP}. We expect that a similar result of {\it loc.cit.} can be generalized to arbitrary $\G$.}
\[
A_{e,i} \prod_{k=1}^r \mathcal{W}_{p, k}^{c_k},
\]
where $\mathcal{W}_{p, k}$ are the partial potential functions in \eqref{partial.potential}. 
Let  $l$ be in $\mathcal{A}^+_{\G^L, \odot}(\mathbb{Z}^t)$.
We have
\[
A_{e,i}^t(w\cdot l)= A_{e,i}^t(l)+ \sum_{k=1}^r c_k \mathcal{W}_{p,k}^t(l) \geq A_{e, i}^t(l).
\]
A similar formula holds for $A_{f,i}^t$.

Let $l_1, l_2 \in \mathscr{A}_{\G^L, \odot}^+(\mathbb{Z}^t)$. Recall the support in \eqref{adcx}. Let us set 
\[
S(l_1,l_2):= \bigcup_{u,v\in W} {\rm Supp}(u\cdot l_1, v\cdot l_2).
\]
By definition, we have
\[
\vartheta_{l_1} \cdot \vartheta_{l_2} = \sum_{l\in S(l_1,l_2)} \widetilde{c}_q(l_1,l_2;l) \theta_{l} = \sum_{l \in \mathscr{A}_{\G^L,\odot}^+(\mathbb{Z}^t) \cap S(l_1,l_2)} \widetilde{c}_q(l_1,l_2;l) \vartheta_{l},
\]
where the last equality follows from \eqref{sum.expr}.
By Lemma \ref{support.lemma}, if $l\in S(l_1,l_2)$, then we have
\[
A_{e,i}^t(l)\geq \min_{u,v\in W}\{A_{e,i}^t(u\cdot l_1)+A_{e,i}^t(v\cdot l_2)\} = A_{e,i}^t(l_1)+A_{e,i}^t(l_2).
\]
By taking the minimum over $A_{e,i}^t$ and $A_{f,i}^t$, we get
\[
F^t(l)\geq F^t(l_1)+ F^t(l_2), \qquad \forall l \in S(l_1,l_2).
\]
Combining with Lemma \ref{tech.lemma}, we conclude the proof of Lemma \ref{tech.lemma1}.
\end{proof}

One key feature of the free $\mathbb{L}$-module $\mathcal{F}_{\lambda}^0$ is that it is of finite dimension. Below we compute the dimension of $\mathcal{F}_{\lambda}^0$. The simple positive roots $\alpha_1, \ldots, \alpha_r$ of $\G$ are the simple positive coroots of $\G^L$. Let us set 
\[
{\alpha}_{{\bf i},k}=s_{i_1}\ldots s_{{i}_{k-1}}(\alpha_{i_k}), \qquad {\alpha}_{{\bf i},k}^\ast =s_{i_1^\ast }\ldots s_{{i}_{k-1}^\ast }(\alpha_{i_k^\ast})
\]
Recall the vector $(\vec{a},\vec{b},\vec{c},\vec{d})$  in \eqref{dimofp}. Let us define
\[
w_{\bf i}(\vec{a},\vec{b},\vec{c},\vec{d})=\left(\sum_{k=1}^n a_k \alpha_{{\bf i}, k} + \sum_{i=1}^{r}(b_i+d_i)\alpha_i,\, \sum_{i=1}^{r}(b_i+d_i)\alpha_{i^\ast} +
    \sum_{k=1}^n d_k \alpha_{{\bf i}, k}^\ast\right) \in X^*(\H)\times X^*(\H).
\]
\bl 
\label{dimofFap}
Let us fix a reduced word ${\bf i}$ of $w_0$.
The dimension $\dim_{\mathbb{L}}\mathcal{F}_\lambda^0$ is equal to the cardinality of the set 
\be
\label{para.set}
\left\{ (\vec{a},\vec{b},\vec{c},\vec{d})\in \mathbb{N}^{2n+2r} ~\middle|~ w_{\bf i}(\vec{a},\vec{b},\vec{c},\vec{d})=
    \lambda
\right\}.
\ee
\el
\begin{proof}
By definition, $\dim_{\mathbb{L}}\mathcal{F}_\lambda^0$ is equal to the cardinality of the set \eqref{para.set2}.

\smallskip 

Let us set $x_i(t)={\rm exp}(te_i)$. Fix a reduced word ${\bf i}=(i_1, \ldots, i_n)$ of $w_0$. Lusztig \cite{LUS90} considered the open embedding
\[
\tau_{\bf i}:~
(\mathbb{G}_m)^n \longrightarrow \U^L, \qquad (t_1,\ldots, t_n) \lms x_{i_n}(t_n)\ldots x_{i_1}(t_1).
\]
The embedding defines a positive structure on $\U^L$, which does not depend on the ${\bf i}$ chosen. 
The tropicalization of $\tau_{\bf i}$ induces a bijection
\[
\tau_{\bf i}^t:~\mathbb{Z}^n \stackrel{\sim}{\lra} \U^L(\mathbb{Z}^t),
\]
whose restriction to $\mathbb{N}^n$ is a bijection 
\[
\tau_{\bf i}^t:~ \mathbb{N}^n \stackrel{\sim}{\lra} \U^L_+(\mathbb{Z}^t).
\]
Let us write ${\bf i}^*=(i_1,\ldots, i_n)$.
Recall $\mathfrak{w}^t$ in \eqref{tropisweight0}.
Composing the map $\left((\tau_{{\bf i}^*}^t)^{-1}, \mathfrak{w}^t, (\tau_{\bf i}^t)^{-1}, \mathfrak{w}^t \right)$  with \eqref{trop.c.a}, we obtain a bijection
\[
p_{\bf i}:\mathscr{A}_{\G^L,\odot}^+(\mathbb{Z}^t)\stackrel{\sim}{\lra} \mathbb{N}^n\times \mathbb{Z}^r \times \mathbb{N}^n\times \mathbb{Z}^r.
\]
By the definition of $F^t$,  a point $l\in \mathscr{A}_{\G^L,\odot}^+(\mathbb{Z}^t)$ satisfies $F^t(l)\geq 0$ if and only if its image
\be
\label{trioc,a.30}
p_{\bf i}^t(l)=(\vec{a},\vec{b},\vec{c},\vec{d})\in \mathbb{N}^{2n+2r}.
\ee

\smallskip 

Recall the edge invariants $h_3$ and $h_4$. The tropicalization of the map 
\[
\mathscr{A}_{\G^L, \odot}\longrightarrow  \H^L\times \H^L,\qquad 
(\mathcal{L}, \{\A_e, \A_f\}\cup\{\A_p\}) \longmapsto (h_3, h_4)
\]
composed with \eqref{tropisweight} is exactly the weight map ${\rm wt}$.

Define $[g]_0=h$ for $g=u_+hu_-$, where $u_{+}\in \U^L$, $h\in \H^L$, and $u_-\in \U^L_-$. It gives rise to a positive map
\[
\beta:~ \U^L\longrightarrow \H^L, \qquad u \longmapsto [\overline{w}_0u]_0.
\]
By Lemma 5.3 of \cite{GS1}, we have 
\[
\beta(x_{i_n}(t_n)\ldots x_{i_1}(t_1))= \prod_{k=1}^n \alpha_{{\bf i},k}(t_k^{-1}).
\]
Let  $h^*=w_0(h^{-1})$. Following Property 4 of \cite[Lemma 6.4]{GS1}, the invariants in Figure \ref{moduliA} satisfy the relations
\[
h_3=\beta(u_1)^*h_1h_2^*,\qquad h_4^{-1}= \beta(u_2)^\ast h_1^*h_2
\]
Let us set 
\[
u_1=x_{i_n^*}(a_n)\ldots x_{i_1^*}(a_1),\qquad h_1=\prod_{i=1}^r \alpha_i(b_i^{-1}),
\]
\[
u_2=x_{i_n}(c_n)\ldots x_{i_1}(c_1),\qquad h_2=\prod_{i=1}^r \alpha_{i^\ast}(d_i^{-1}).
\]
Then we have
\[
h_3^{-1}=\prod_{k=1}^n \alpha_{{\bf i},k}(a_k) \cdot \prod_{i=1}^r \alpha_i(b_id_i),\qquad h_4^{-1}= \prod_{k=1}^n \alpha_{{\bf i},k}^\ast(c_k) \cdot \prod_{i=1}^r \alpha_{i^*}(b_id_i).
\]
whose tropicalization composed with \eqref{tropisweight} recovers the map in \eqref{para.set}.  
Therefore for any  $l\in \mathscr{A}^+_{\G^L,\odot}(\mathbb{Z}^t)$, we have
\[
{\rm wt}(l)=\lambda   ~~\Longleftrightarrow w_{\bf i}(p_{\bf i}(l)) =\lambda.
\]
Therefore $p_{\bf i}$ defines a bijection from the set \eqref{para.set} to the set \eqref{para.set2}. The Lemma follows. 
\end{proof}

\subsection{An embedding of the quantum group} \la{sec5.4}
By Lemma \ref{lemma12d}, the algebra $\mathcal{O}_q(\mathscr{P}_{\G, \odot})^W$ is a free $\mathbb{L}$-module. We shall also change the basis and consider
\[
\mathcal{O}_q(\mathscr{P}_{\G, \odot})^W_{\mathbb{K}}:= \mathcal{O}_q(\mathscr{P}_{\G, \odot})^W\otimes_{\mathbb{L}}\mathbb{K}.
\]
The paper \cite{GS3} constructs a natural embedding
\be
\label{kappa}
\begin{array}{c}
\kappa:~\mathcal{D}_q(\mathfrak{b}) \longrightarrow \mathcal{O}_q(\mathscr{P}_{\G, \odot})^W_{\mathbb{K}}\\
{\bf E}_i \longmapsto \mathbb{W}_{1, i}, \quad {\bf F}_i \longmapsto \mathbb{W}_{2, i^*}, \quad {\bf K}_i \longmapsto  \mathbb{K}_{1,i},\quad 
\widetilde{\bf K}_i \longmapsto \mathbb{K}_{2, i^*},
\end{array}
\ee
where $i^*$ is such that $s_{i^*}= w_0 s_i w_0$.

\smallskip

Let us write 
\[
\mathcal{F}^0:= \mathcal{F}^0\OPW=\bigoplus_\lambda \mathcal{F}_\lambda^0.
\]
By Lemma \ref{tech.lemma1}, $\mathcal{F}^0$ forms a $\mathbb{L}$-subalgebra of $\OPW$ with a basis 
\[
\left\{ \vartheta_l ~\middle|~ l\in \mathscr{A}^+_{\G^L,\odot}(\mathbb{Z}^t),~ F^t(l)\geq 0\right\}.
\]
We further set 
\[
\mathcal{F}^0_{\mathbb{K}}:= \mathcal{F}^0 \otimes_{\mathbb{L}} \mathbb{K}, \qquad \mathcal{F}^0_{\alpha, \mathbb{K}}:= \mathcal{F}_{\alpha}^0 \otimes_{\mathbb{L}} \mathbb{K}.
\]
Recall that $\widetilde{\bf U}_q(\mathfrak{g})$ is the $\mathbb{L}$-linear span of the PBW basis. Below we investigate the image of $\widetilde{\bf U}_q(\mathfrak{g})$ under the map $\kappa$.

\smallskip 

The braid group ${\rm Br}_{\G}$ associated with $\G$ is generated by $\sigma_1, \ldots, \sigma_r$, satisfying the relations
\[\left\{
\begin{array}{ll}
\sigma_i\sigma_j\sigma_i=\sigma_j\sigma_i\sigma_j, & \mbox{if $a_{ij}=-1$};\\
\sigma_i\sigma_j=\sigma_i\sigma_j, &\mbox{if $a_{ij}=0$}.
\end{array}\right.
\]
The paper \cite{GS3} shows that for every boundary circle of $\bS$ with an even number of marked points, there is a Braid group action on the moduli space $\mathscr{P}_{\G, \bS}$. Furthermore, the group ${\rm Br}_{\G}$ acts via quasi-cluster transformations (cf. Theorem 13.13 of {\it loc.cit.}), and therefore can be lift to an action on $\mathcal{O}_q(\mathscr{P}_{\G, \bS})$. For $\mathscr{P}_{\G, \odot}$, an explicit expression of the braid group action has been given in \cite[\S 4.3]{Sh2}. Theorem 15.14 of \cite{GS3} shows that the braid group action on $\mathcal{O}_q(\mathscr{P}_{\G, \odot})$ coincides with Lusztig's braid group action on $\mathcal{D}_q(\mathfrak{b})$ under the isomorphism \eqref{kappa}. 

Recall ${\bf E}_{{\bf i},k}$ and ${\bf F}_{{\bf i},k}$ in \eqref{pbw.c1}. By Theorem \ref{basic.theta.q}, we have
\be
\label{2ascm}
\kappa({\bf E}_{{\bf i},k})= \kappa(T_{i_1}T_{i_2}\ldots T_{i_{k-1}}{\bf E}_{i_k}) =\sigma_{i_1}\sigma_{i_2}\ldots \sigma_{i_{k-1}}\cdot \theta_{l_i}.
\ee
Since $\sigma_i$ acts via quasi-cluster transformation, we see that $\kappa({\bf E}_{{\bf i},k})$ remains a quantum theta function. In particular
\[
\kappa({\bf E}_{{\bf i},k}) \in \mathcal{O}_q(\mathscr{P}_{\G,\odot})^W. 
\]
\begin{lemma}
\label{lemm59}
We have 
\be
\label{weight.e.f}
\kappa({\bf E}_{{\bf i},k})\in \mathcal{F}_{(\alpha_{{\bf i},k},0)}^0, \qquad \kappa({\bf F}_{{\bf i},k})\in \mathcal{F}_{(0, \alpha_{{\bf i},k}^\ast)}^0
\ee
\be
\label{weight.k}
\kappa({\bf K}_i)\in \mathcal{F}_{(\alpha_i,\alpha_{i^\ast})}^0, \qquad \kappa(\widetilde{\bf K}_i) \in \mathcal{F}_{(\alpha_i, \alpha_{i^\ast})}^0.
\ee
\end{lemma}
\begin{proof} Recall the tropical points $l_i, l_i'$ associated with the frozen vertices. By definition, we have
\[
\alpha(l_i) =(\alpha_i, 0), ~~~\alpha(l_i')=(0, \alpha_{i^\ast}),~~~\mbox{and } F^t(l_i)=F^t(l_i')=0.
\]
 By Theorem \ref{basic.theta.q}, we have 
\be
\label{checka}
\kappa({\bf E}_i)=\vartheta_{l_i} \in \mathcal{F}_{(\alpha_i,0)}^0, \qquad \kappa({\bf F}_{i})=\vartheta_{l_{i^\ast}'}\in \mathcal{F}_{(0,\alpha_{i^\ast})}^0
\ee
The relation \eqref{eq3} implies that 
\[
\kappa({\bf E}_i{\bf F}_i-{\bf F}_i{\bf E}_i)=\kappa\left((q-q^{-1})({\bf K}_i-\widetilde{\bf K}_i)\right) = (q-q^{-1}) (\vartheta_{k_i}-\vartheta_{k_{i^\ast}}) \in \mathcal{F}^0_{(\alpha_{i},\alpha_{i^\ast})}.
\]
Therefore we get \eqref{weight.k}.

Note that ${\bf E}_{{\bf i}, k}$ belongs to $\mathcal{U}(\mathfrak{n})$, the $\mathbb{K}$-subalgebra generated by ${\bf E}_i$ for $1\leq i\leq r$. Therefore $\kappa({\bf E}_{{\bf i},k})$ belong to the $\mathbb{K}$-subalgebra generated by $\vartheta_{l_i}$ for $1\leq i\leq r$. Combining \eqref{2ascm} with \eqref{checka}, we get 
\be
\label{iacss}
{\bf E}_{{\bf i}, k} \in \mathcal{F}^0.
\ee
Now let us show that 
\be
\label{oncinc}
T_i{\bf E}_j \in \OPW_{(s_i(\alpha_j),0)}.
\ee
Indeed, if $a_{ij}=-1$, then 
\[
\kappa(T_i{\bf E}_j) =\kappa\left(\frac{q^{1/2}\bE_j\bE_i - q^{-1/2}\bE_i\bE_j}{q-q^{-1}}\right)= \frac{q^{1/2}\vartheta_j\vartheta_i - q^{-1/2}\vartheta_i\vartheta_j}{q-q^{-1}} \in \mathcal{F}^0_{(s_i(\alpha_j),0)}.
\]
For $a_{ij}=0$ or $i=j$, the inclusion \eqref{oncinc} follows by a similar calculation. 
By induction on $k$, we get 
\[
{\bf E}_{{\bf i}, k} \in \OPW_{(\alpha_{{\bf i}, k},0)}.
\]
Combining with \eqref{iacss}, we obtain the desired inclusion. The inclusion for ${\bf F}_{{\bf i}, k}$ follows by a similar argument.
\end{proof}

Recall the PBW basis elements in \eqref{PBW.as12}. 
Fix a $\lambda \in X^*(\H)\times X^*(\H)$. 
Let $\widetilde{\bf U}_q(\mathfrak{g})_\lambda$ be the $\mathbb{L}$-linear span of the PBW basis elements $\bE_{\bf i}(\vec{a})\bK(\vec{b})\bF_{\bf i}(\vec{c})\tbK(\vec{d})$ such that 
$w_{\bf i}(\vec{a},\vec{b},\vec{c},\vec{d})=\lambda$.
By Lemma \ref{lemm59}, we have
\be
\label{PBW011}
\kappa(\widetilde{\bf U}_q(\mathfrak{g})_\lambda) \subset \mathcal{F}_\lambda^0.
\ee

\smallskip 
\begin{proposition}
\label{propa,a}
 The map $\kappa: \widetilde{\bf U}_q(\mathfrak{g})_\lambda \rightarrow  \mathcal{F}_\lambda^0 $  is an isomorphism. 
\end{proposition}
\begin{proof} By Lemma \ref{dimofFap}, we have
\[
\dim_{\mathbb{L}} \widetilde{\bf U}_q(\mathfrak{g})_\lambda = \dim_{\mathbb{L}} \mathcal{F}_{\lambda}^0 < \infty.
\]
The map $\kappa$ is injective, therefore we get 
\[
\kappa\left(\widetilde{\bf U}_q(\mathfrak{g})_\lambda\otimes_{\mathbb{L}}\mathbb{K}\right) = \mathcal{F}_{\alpha}^0 \otimes_{\mathbb{L}}\mathbb{K}.
\]
Therefore every element $\phi\in \mathcal{F}_\alpha^0$ admits a finite $\mathbb{K}$-linear expansion of the images of the PBW basis elements.
Now let us fix a quantum seed for $\mathcal{O}_q(\mathscr{P}_{\G, \odot})$. 
Recall that $\kappa({\bf E}_{{\bf i},k}), \kappa({\bf K}_i), \kappa({\bf F}_{{\bf i},k})$, and  $\kappa(\widetilde{{\bf K}}_i)$ are all quantum theta functions. By \eqref{sann}, we have
\[
\kappa(\bE_{\bf i}(\vec{a})\bK(\vec{b})\bF_{\bf i}(\vec{c})\tbK(\vec{d})) = q^{s} X_{v}+ \sum_{w>v} d(q) X_w,
\]
where $d(q)$ is a Laurent polynomial of $q^{\frac{1}{2}}$ with integral coefficients. 
As a consequence, the coefficients of the $\mathbb{K}$-linear expansion of $\phi$ must be in $\mathbb{L}$. The Proposition follows.  
\end{proof}

\begin{theorem} 
\label{iso,mai}
The restriction of the embedding $\kappa$ to $\widetilde{{\bf U}}_q(\mathfrak{g})$ gives rise to an  isomorphism 
\[
\kappa: ~ \widetilde{{\bf U}}_q(\mathfrak{g}) \stackrel{\sim}{\longrightarrow} \mathcal{F}^0.
\]
\end{theorem}

\begin{proof}

 As a direct consequence of Proposition \ref{propa,a}, we see that the map $\kappa$ from $\widetilde{{\bf U}}_q(\mathfrak{g})$ to $\mathcal{F}^0$ is a bijection. Note that $\mathcal{F}^0$ is an $\mathbb{L}$-algebra. As a consequence,  $\widetilde{\bf U}_q(\mathfrak{g})$ is an $\mathbb{L}$-algebra. 
\end{proof}
\smallskip

\bt 
\label{isomo.duan.ao}
The map $\kappa$ defines an algebra isomorphism
\be
\label{zk1}
\kappa:~ {\bf D}_q(\mathfrak{b})\stackrel{\sim}{\lra} \mathcal{O}_q(\mathscr{P}_{\G, \odot})^W.
\ee
\et
\begin{proof} Let us set 
\be
\label{asNOJ}
\mathbb{O}_i:= \mathbb{K}_{1,i}\mathbb{K}_{2,i^*}=\theta_{k_i}\theta_{k_{i^*}'}.
\ee
As proved in \cite{GS3}, the element $\mathbb{O}_i$ is an Casimir element. By Lemma \ref{lemma.case12}, we see that 
\[
\kappa({\bf K}_{i}^{-1}) = \mathbb{K}_{2,i^*} \mathbb{O}_i^{-1}, \qquad \kappa(\widetilde{\bf K}_{i}^{-1})= \mathbb{K}_{1,i} \mathbb{O}_i^{-1}
\]
are quantum theta functions contained in $\mathcal{O}_q(\mathscr{P}_{\G,\odot})^W$. Therefore the injective map $\kappa$ takes ${\bf D}_q(\mathfrak{b})$ into $\mathcal{O}_q(\mathscr{P}_{\G,\odot})^W$.

It remains to show that the map $\kappa$ is surjective. Let us define
\[
\mathbb{O}=\prod_{i=1}^r \mathbb{O}_i.
\]
Note that $\mathbb{O}$ is a Casimir element. Recall the filtration $\mathcal{F}^n$ of $\mathcal{O}_q(\mathscr{P}_{\G,\odot})^W$.
By looking at the leading term of $\mathbb{O}$, we see that for any $n\in \mathbb{Z}$, the power
\[
\mathbb{O}^{-n} \in \mathcal{F}^{n}.
\]
By Lemma \ref{tech.lemma1}, we see that
\[
\mathcal{F}^0\cdot \mathbb{O}^{-n} = \mathcal{F}^{n}.
\]

Now let us set 
\[
{\bf O}_i:= {\bf K}_{i}\widetilde{\bf K}_{i}, \qquad 
{\bf O}=\prod_{i=1}^r {\bf O}_{i}.
\]
We have $\kappa({\bf O})=\mathbb{O}$. By the definition of ${\bf D}_q(\mathfrak{b})$, we see that ${\bf D}_q(\mathfrak{b})$ is closed under the multiplication of ${\bf O}^{-n}$. In other words, we have
\[
\widetilde{\bf U}_q(\mathfrak{g})\cdot {\bf O}^{-n}\subset {\bf D}_q(\mathfrak{g}).
\]
Applying the map $\kappa$ on both sides, by Theorem \ref{iso,mai}, we get 
\[
\mathcal{F}^n \subset \kappa\left({\bf D}_q(\mathfrak{b})\right), \qquad \forall n\in \mathbb{Z}. 
\]
Therefore the map $\kappa$ is a surjection.
\end{proof}
 
\bd
Let $\mathcal{I}$ be an ideal of $\mathcal{O}_q(\mathscr{P}_{\G, \odot})^W$ generated by
\[
\mathbb{O}_i-1, \qquad \forall 1\leq i \leq r.
\]
Denote by $\mathcal{O}_q(\mathscr{P}_{\G, \odot})^W{\Big \slash} \mathcal{I}$.
\ed

Recall that ${\bf U}_q(\mathfrak{g})$ is a quotient algebra of ${\bf D}_q(\mathfrak{b})$ obtained by imposing the conditions $\mathbf{O}_i=1$ for $1\leq i\leq r$. The following result is  a direct consequence of Theorem \ref{isomo.duan.ao}.

\bt 
\label{isomoacs.quan}
The map $\kappa$ descends to an $\mathbb{L}-$algebra isomorphism 
\be 
\kappa:~
{\bf U}_q(\mathfrak{g}) \stackrel{\sim}{\lra}\mathcal{O}_q(\mathscr{P}_{\G, \odot})^W{\Big \slash} \mathcal{I}.
\ee
\et

The paper \cite{GS3} shows that $\kappa$ is a Hopf algebra homomorphism. Combining with Theorem \ref{isomoacs.quan}, we prove Theorem \ref{main1}.

\subsection{Cluster canonical basis of ${\bf U}_q(\mathfrak{g})$} Through the isomorphism $\kappa$, the set 
\be
\label{q.basis.z}
{\bf \Theta}:=\left\{\vartheta_l~\middle|~ l\in \mathscr{A}^+_{\G^L,\odot}(L)
\right\}
\ee
provides a $\mathbb{L}$-linear basis of  ${\bf D}_q(\mathfrak{b})$ with structural coefficients in $\mathbb{N}[q^{\frac{1}{2}}, q^{-\frac{1}{2}}]$.

\smallskip

Let ${C}:=\{c_1, \ldots, c_r\}$ be the tropical points that parametrize the $W$-inavariant Casimir elements $\mathbb{O}_1, \ldots, \mathbb{O}_r$ in \eqref{asNOJ}.  In other words, we have
\[
\mathbb{O}_i=\theta_{c_i}=\vartheta_{c_i}, \qquad \forall~ 1\leq i\leq r.
\]
Following \eqref{ascjna}, the set $C$ gives rise to a $\mathbb{Z}^r$ action on the set $\mathscr{A}_{\G^L, \odot}(\mathbb{Z}^t)$. 
Recall the set ${\rm Or}_{C}(\mathscr{A}_{\G^L, \odot}(\mathbb{Z}^t))$ 
of $\mathbb{Z}^r$-orbits of $\mathscr{A}_{\G^L, \odot}(\mathbb{Z}^t)$ under the above $\mathbb{Z}^r$ action.

\bl 
There is a natural bijection that identifies  ${\rm Or}_{C}(\mathscr{A}_{\G^L,\odot}(\mathbb{Z}^t))$ with the set 
\[
\U^L(\mathbb{Z}^t)\times X^*({\rm H}) \times \U^L(\mathbb{Z}^t).
\]
\el

\begin{proof}
By the bijection \eqref{acnoanbit}, we get a $\mathbb{Z}^r$ action on the set  
\be 
\label{adcniaj}
\U^L(\mathbb{Z}^t)\times X^*(\H) \times \U^L(\mathbb{Z}^t) \times X^*(\H).
\ee
Under this action, every  $a=(a_1, \ldots, a_r)\in \Z^r$ maps  an element $ (x, \lambda, y, \mu)$ in \eqref{adcniaj} to
\[
\left(x, \lambda+\lambda_a, y, \mu-w_0(\lambda_a)\right), \qquad \mbox{where } \lambda_a:=\sum_{i=1}^r a_i\alpha_i. 
\]
In particular, there is a unique $a$ such that $\mu-w_0(\lambda_a)=0$. Therefore, each $\mathbb{Z}^r$-orbit contains a unique representative such that the last coordinate $\mu =0$. 
\end{proof}

Following the same argument in the above proof and the bijection \eqref{trop.c.a}, we obtain a bijection of the sets
\be
\label{p,quant,a}
{\rm Or}_{C}(\mathscr{A}_{\G^L,\odot}^+(\mathbb{Z}^t))\stackrel{\sim}{=} \U^L_+(\mathbb{Z}^t)\times X^*(\H)\times \U^L_+(\mathbb{Z}^t). 
\ee
Meanwhile, for each $l\in {\rm Or}_{C}(\mathscr{A}_{\G^L,\odot}^+(\mathbb{Z}^t))$, by Lemma \ref{lemma.case12}, we get  
\[
\vartheta_{l}\cdot \vartheta_{c_i}=\vartheta_{l+c_i}
\]
It gives rise to a $\mathbb{Z}^r$ action on the basis \eqref{q.basis.z}.
After quotient out the ideal $\mathcal{I}$, the functions $\vartheta_l$ in the same $\mathbb{Z}^r$ orbit descends to one element in  ${\bf U}_q(\mathfrak{g})$. 
Putting them together, we obtain a basis $\overline{\bf \Theta}$ for ${\bf U}_q(\mathfrak{g})$, parametrized by the sets \eqref{p,quant,a}.

\smallskip 

\paragraph{{\bf Proof of Theorem \ref{main.theorem.baisi.p}}.}
We have already proved part d).

Part a) is a direct consequence of Lemma \ref{check2}.

The braid group ${\rm Br}_{\mathfrak{g}}$, the outer automorphism group ${\rm Out}(\G)$, and the Weyl group $W$ act on quantum cluster algebra $\mathcal{O}_q(\mathscr{P}_{\G, \odot})$ via quasi cluster automorphisms \cite[\S 13]{GS3}.  Following Theorem \ref{quantum.cluster.dual}, these three groups preserve the quantum theta basis $\Theta(\mathcal{O}_q(\mathscr{P}_{\G, \odot}))$ as a set. The actions of ${\rm Br}_{\mathfrak{g}}$ and ${\rm Out}(\G)$ are compatible with the Weyl group action. Therefore actions of  ${\rm Br}_{\mathfrak{g}}$ and ${\rm Out}(\G)$ preserves ${\bf \Theta}$ as a set. Furthermore, if two functions $\vartheta_{l_1}$ and $\vartheta_{l_2}$ are the same $\mathbb{Z}^r$-orbit, then their image under the action of any element in  ${\rm Br}_{\mathfrak{g}}$ or ${\rm Out}(\G)$ are also in the same $\mathbb{Z}^r$-orbit. Therefore, they preserve $\overline{\bf \Theta}$ as a set.
Through the isomorphism $\kappa$, the braid group action coincides with Lusztig's braid group action on ${\bf U}_q(\mathfrak{g})$, and the ${\rm Out}(\G)$ coincides with the Dynkin automorphisms. 
As a consequence, we prove part b) of Theorem  \ref{main.theorem.baisi.p}.

Part c) of Theorem \ref{main.theorem.baisi.p} follows from 4) of Theorem \ref{quantum.cluster.dual}. 
\bibliographystyle{amsalpha-a}

\bibliography{biblio}



\end{document}